\documentclass[11pt,amssymb]{amsart}
\usepackage{amssymb}
\usepackage[mathscr]{eucal}

\newcommand{\Z}{{\mathbb Z}}
\newcommand{\C}{{\mathbb C}}
\newcommand{\Q}{{\mathbb Q}}
\newcommand{\R}{{\mathbb R}}
\newcommand{\fR}{{\mathfrak R}}

\newcommand{\I}{{\mathcal O}}

\newcommand{\bG}{{\mathbb G}}

\newcommand{\G}{{\mathcal G}}
\newcommand{\g}{{\mathfrak{g}}}

\newcommand{\fa}{{\mathfrak{a}}}

\newcommand{\B}{{\mathscr{B}}}

\newcommand{\Hom}{{\rm Hom}}

\newcommand{\ba}{\mbox{\boldmath{$\alpha$}}}
\newcommand{\bm}{\mbox{\boldmath{$\mu$}}}

\newtheorem{thm}{Theorem}[subsection]
\newtheorem{lemma}[thm]{Lemma}
\newtheorem{prop}[thm]{Proposition}
\newtheorem{cor}[thm]{Corollary}

\newcommand{\St}{{\rm St}}

\usepackage[all,cmtip]{xy}
\begin{document}
\title[Abstract homomorphisms and applications]{Abstract homomorphisms of algebraic groups and applications}



\begin{abstract}
We give a survey of recent rigidity results for linear representations of elementary subgroups of Chevalley groups over commutative rings, and of other similar groups, and also discuss some applications to deformations of representations of elementary groups over finitely generated commutative rings.
\end{abstract}

\author[I.A.~Rapinchuk]{Igor A. Rapinchuk}

\address{Department of Mathematics, Yale University, New Haven, CT 06502}

\email{igor.rapinchuk@yale.edu}

\maketitle

\section{Introduction}\label{S:I}

This paper is an expanded version of the author's talk at the Conference on Group Actions and Applications in Geometry, Topology, and Analysis held in Kunming. Our main goal is to give an overview of recent rigidity results for abstract representations of elementary subgroups of Chevalley groups over commutative rings that were proved using algebraic rings and some results from algebraic $K$-theory. We also describe applications of these results to the study of character varieties of elementary subgroups over finitely generated commutative rings.

We would like to begin with a brief outline of the general philosophy in the study of abstract homomorphisms between the groups of rational points of algebraic groups. Suppose
$G$ and $G'$ are algebraic groups that are defined over infinite fields $K$ and $K',$ respectively. Let
$$
\varphi \colon G(K) \to G'(K')
$$
be an abstract homomorphism between their groups of rational points. Then, under appropriate assumptions, one
expects to be able to write $\varphi$ essentially as a composition $\varphi = \beta \circ \alpha,$ where $\alpha \colon G(K) \to _{K'}\!G(K')$ is induced by a field homomorphism $\tilde{\alpha} \colon K \to K'$ (and $_{K'}\!G$ is the group obtained from $G$ by base change via $\tilde{\alpha}$), and $\beta \colon _{K'}\!G(K') \to G'(K')$ arises from a $K'$-defined \emph{morphism of algebraic groups} $_{K'}\!G \to G'$.
Whenever $\varphi$ admits such a decomposition, we will say that it has a {\it standard description}, and we will use the term {\it rigidity statement} to refer to any result that asserts that all members in a certain class of abstract homomorphisms admit a standard description.

The study of abstract homomorphisms was initiated in the first part of the $20^{{\rm th}}$ in the work of O.~Schreier and B.L.~van der Waerden \cite{SvdW} on abstract automorphisms of projective special linear groups over fields, and that of E.~Cartan \cite{C} on abstract homomorphisms of semisimple Lie groups into compact groups.
Afterwards, there was a great deal of activity in describing
the abstract automorphisms of various classical groups. One of the landmarks was the work of Dieudonn\'e \cite{D},
which contains a number of results for classical groups not only over fields, but also over division rings.
Later, his ideas were further developed by O'Meara and his collaborators (see \cite{HO} for a comprehensive discussion of various rigidity statements in this context).
Since our focus in this survey will not specifically be on classical groups, we will not give precise statements of these result, but would only like to point out that they are all consistent with the general philosophy outlined above.

In his 1967 {\it Lectures on Chevalley Groups} \cite{Stb1}, Steinberg shifted the focus from homomorphisms and automorphisms of classical groups, where results were generally proved on a case-by-case basis using {\it ad hoc} techniques, to a systematic study of abstract homomorphisms in the context of general semisimple algebraic groups;
one of his results was a standard description for {\it abstract automorphisms} of (the groups of points of) universal Chevalley groups over perfect fields (see \cite{Stb1}, Theorem 30, pg. 150). The proof was based on an analysis of the action of an abstract group automorphism on the canonical unipotent one-parameter root subgroups.


Several years later, Steinberg's work was substantially generalized by Borel and Tits in their fundamental paper \cite{BT}. They considered the following situation: suppose $G$ and $G'$ are algebraic groups defined over infinite fields $k$ and $k'$, respectively, with $G$ absolutely almost simple $k$-isotropic and $G'$ absolutely simple adjoint. Denote by $G^+$ the subgroup of $G(k)$ generated by the $k$-points of split (smooth) connected unipotent $k$-subgroups. Then, they showed that any abstract homomorphism $\rho \colon G(k) \to G'(k')$ such that $\rho (G^+)$ is Zariski-dense in $G'(k')$ has a standard description (see \cite{BT}, Theorem A, for a precise statement). In fact, they
obtained a similar result in the case that $G'$ is only assumed to be reductive.

On the other hand, Borel and Tits noted that the reductiveness of $G'$ is (in general) not a consequence of the Zariski-density of $\varphi (G^+)$ and that without this assumption, $\rho$ may fail to have the above description (\cite{BT}, 8.18). They considered the following example. Suppose $G$ is an absolutely almost simple group defined over an infinite field $k$ (e.g., $G = SL_n$), and let $K$ be a field extension of $k$ which admits a nontrivial $k$-derivation $\delta \colon K \to K$ (for example, one can take $\delta$ to be the usual operator of differentiation on the field of rational functions $K = k(x)$).
Let $\g$ be the Lie algebra of $G$ and set
$$
G' = \g \ltimes G,
$$
with $G$ acting on $\g$ via the adjoint representation. Next, define a group homomorphism $\alpha \colon G(K) \to G'(K)$ by
$$
G(K) \ni g \mapsto (g^{-1} \cdot \Delta(g), g) \in G'(K),
$$
where $\Delta$ denotes the map induced by the derivation $\delta$ (for example, if $G = SL_n$, then $\Delta$ is obtained by simply applying $\delta$ to the matrix entries). Since $K$ is infinite and $\delta$ is a nontrivial derivation, it is easy to see that $\alpha$ has Zariski-dense image, and, furthermore, the unipotent radical of $G'$ is $\g$.

This example can be interpreted more conceptually as follows. Let $A = K[\varepsilon]$, where $\varepsilon^2 = 0.$ It is clear that the map
$$
f \colon K \to A, \ \ \ x \mapsto x + \delta(x) \varepsilon
$$
is a homomorphism of $k$-algebras, and hence induces a group homomorphism $F \colon G(K) \to G(A)$ on the groups of rational points. On the other hand, we have the well-known identification (cf. \cite{Bo}, 3.20)
$$
G(A) \stackrel{t}{\longrightarrow} \g(K) \ltimes G(K) = G'(K), \ \ \ X + Y \varepsilon \mapsto (X^{-1} \cdot Y, X).
$$
One then easily checks that $\alpha = t \circ F$. Thus, we see that $\alpha$ essentially arises from a {\it homomorphism of $k$-algebras}.

The nature of this example led Borel and Tits to formulate the following conjecture (see \cite{BT}, 8.19):

\vskip3mm

\noindent (BT) \ \ \parbox{14cm}{Let $G$ and $G'$ be algebraic groups defined over infinite fields $k$ and $k'$, respectively. If $\rho \colon G(k) \to G'(k')$ is any abstract homomorphism such that $\rho(G^+)$ is Zariski-dense in $G'(k'),$ then {\it there exists a commutative finite-dimensional $k'$-algebra $B$ and a ring homomorphism $f \colon k \to B$ such that  $$\rho \vert_{G^+} = \sigma \circ r_{B/k'} \circ F$$ where $F \colon G(k) \to _{B}\!G(B)$ is induced by $f$ ($_{B}\!G$ is the group obtained by change of scalars), $r_{B/k'} \colon _{B}\!G(B) \to R_{B/k'} (_{B}\!G)(k')$ is the canonical isomorphism  (here $R_{B/k'}$ denotes the functor of restriction of scalars), and $\sigma$ is a rational $k'$-morphism of $R_{B/k'} (_{B}\!G)$ to $G'.$}}

\vskip3mm

\noindent{\bf Remark.} It was pointed out to us by B.~Conrad and G.~Prasad that, using techniques from the theory of pseudo-reductive groups (developed in \cite{CGP}, Chapter 9), one can construct counterexamples to (BT) over all local and global function fields of characteristic 2 (or, more generally, over any field $k$ of characteristic 2 such that $[k:k^2] = 2$). The groups that arise in these counterexamples are perfect and $k$-simple. So, one should exclude fields of characteristic 2 in
the statement of (BT).
\vskip3mm




Shortly after the conjecture was formulated, Tits \cite{T} sketched a proof of (BT) in the case that $k = k' = \R.$ Prior to our work, the only other available result was due to L.~Lifschitz and A.S.~Rapinchuk \cite{LR}, where the conjecture was essentially proved in the case that $k$ and $k'$ are fields of characteristic 0, $G$ is a universal Chevalley group, and $G'$ is an algebraic group with commutative unipotent radical.

While the above results only deal with abstract homomorphisms of groups of points over {\it fields}, it should be pointed out that there has also been considerable interest and activity in analyzing
abstract homomorphisms of higher rank arithmetic groups and lattices. For example,
Bass, Milnor, and Serre \cite{BMS} used their solution of the congruence subgroup problem for $G = SL_n (n \geq 3)$ and $Sp_{2n} (n \geq 2)$ to prove that any representation $\rho \colon G(\Z) \to GL_m (\C)$ coincides on a subgroup of finite index $\Gamma \subset G(\Z)$ with the restriction of some rational morphism $\sigma \colon G(\C) \to GL_m(\C)$ of algebraic groups. A few years later,
Serre \cite{Serre} established a similar result for the group $SL_2 (\Z[1/p])$.
(In fact, the results of \cite{BMS} and \cite{Serre} apply to the groups of points over rings of $S$-integers in arbitrary number fields, but their statements in those cases are a bit more technical).
The most general results about representations of higher rank arithmetic groups and lattices are contained in Margulis's Superrigidity Theorem (cf. \cite{Mar}, Chap. VII). On the other hand, Steinberg \cite{Stb2} showed that the above results for representations of $SL_n (\Z),$ with $n \geq 3$, can be derived directly from the commutator relations for elementary matrices.

However, relatively little was previously known about abstract homomorphisms of groups of points over {\it general commutative rings}, which has been the primary focus of our work in this area. We should point out that our methods are inspired, to some degree, by Steinberg's generators-relations approach.

Before formulating our results, we need to fix some notations and recall a few definitions concerning Chevalley-Demazure group schemes. For details on the latter, we refer the reader to Borel's article \cite{Bo1}; a systematic development of the theory of split reductive group schemes over arbitrary nonempty schemes can be found in B.~Conrad's notes \cite{BCRed}. Let $\Phi$ be a reduced irreducible root system of rank $\geq 2$ and let $G$ be the corresponding universal Chevalley-Demazure group scheme over $\Z$ (i.e., $G$ is an affine group scheme over $\Z$ which is reduced and of finite type, such that for any algebraically closed field $L$, the group of rational points $G(L)$ coincides with the usual simply connected Chevalley group $G_L$ over $L$ with root system $\Phi$).
Then, in particular, for every root $\alpha \in \Phi$, we have a canonical morphism of group schemes $e_{\alpha} \colon \mathbb{G}_a \to G$, where $\mathbb{G}_a = {\rm Spec} \  \Z[T]$ is the standard additive group scheme over $\Z.$ For any commutative ring $R$, the group of $R$-points $G(R)$ is usually referred to as the {\it universal Chevalley group of type $\Phi$ over $R$}.
For $\alpha \in \Phi$, the morphism $e_{\alpha}$ induces a group homomorphism $R^+ \to G(R)$, which will also be denoted $e_{\alpha}$ (rather than $(e_{\alpha})_R$) whenever this does not cause confusion. Then $e_{\alpha}$ is an isomorphism between $R^+$ and the subgroup $U_{\alpha} (R) := e_{\alpha} (R)$ of $G(R).$ The subgroup of $G(R)$ generated by the $U_{\alpha} (R)$, for all $\alpha \in \Phi$, will be denoted by $G(R)^+$ and called the {\it elementary subgroup of $G(R)$}. In our results, we will usually need to assume that $(\Phi, R)$ is a {\it nice pair}: namely, $2 \in R^{\times}$ if $\Phi$ contains a subsystem of type $B_2$ and $2, 3 \in R^{\times}$ if $\Phi$ is of type $G_2.$

Our first result is a rigidity statement for abstract representations $\rho \colon G(R)^+ \to GL_n (K)$, where $K$ is an algebraically closed field. In the statement below, for a finite-dimensional commutative $K$-algebra $B$, we view the group of rational points $G(B)$ as an algebraic group over $K$ using the functor of restriction of scalars (see \S \ref{S:Rat} for further details).

\vskip2mm

\noindent {\bf Theorem 1.} (\cite{IR}, Main Theorem) {\it Let $\Phi$ be a reduced irreducible root system of rank $\geq 2$, $R$ a commutative ring such that $(\Phi, R)$ is a nice pair, and $K$ an algebraically closed field. Assume that $R$ is noetherian if $\mathrm{char} \: K > 0.$ Furthermore let $G$ be the universal Chevalley-Demazure group scheme of type $\Phi$ and let $\rho \colon G(R)^+ \to GL_n (K)$ be a finite-dimensional linear representation over $K$ of the elementary subgroup $G(R)^+ \subset G(R)$. Set $H = \overline{\rho (G(R)^+)}$ (Zariski closure), and let $H^{\circ}$ denote the connected component of the identity of $H$. Then in each of the following situations
\vskip1mm

\noindent {\rm (1)} $H^{\circ}$ is reductive;

\vskip1mm

\noindent {\rm (2)} $\mathrm{char} \: K = 0$ and $R$ is semilocal;

\vskip1mm

\noindent {\rm (3)} $\mathrm{char} \: K = 0$ and the unipotent radical $U$ of $H^{\circ}$ is commutative,

\vskip1mm

\noindent there exists a commutative finite-dimensional $K$-algebra $B$, a ring homomorphism $f \colon R \to B$ with Zariski-dense image, and a morphism $\sigma \colon G(B) \to H$ of algebraic $K$-groups such that for a suitable subgroup $\Delta \subset G(R)^+$ of finite index, we have
$$
\rho \vert_{\Delta} = (\sigma \circ F) \vert_{\Delta},
$$
where $F \colon G(R)^+ \to G(B)^+$ is the group homomorphism induced by $f$.}

\vskip2mm

\noindent Notice that if $R = k$ is a field of characteristic $\neq 2$ or 3, then $R$ is automatically semilocal and $(\Phi, R)$ is a nice pair, so Theorem 1 provides a proof of Conjecture (BT) in the case that $G$ is split and $k' = K$ is an algebraically closed field of characteristic zero.

One of the essential ingredients in the proof of Theorem 1 is the use of the notion of an {\it algebraic ring}, i.e. an affine algebraic variety endowed with the structure of a ring defined by regular maps
(see \S \ref{S:AR}). More precisely, by extending a construction initially carried out by Kassabov and Sapir \cite{Kas} for $SL_n$, we show that if $(\Phi, R)$ is a nice pair and $G$ is a universal Chevalley group scheme, then to any representation $\rho$ of $G(R)^+$, one can associate an algebraic ring $A$, together with a ring homomorphism $f \colon R \to A$ having Zariski-dense image. The construction of $A$ relies on explicit computations with Steinberg commutator relations in $G(R)^+.$ Furthermore, a structural result that is needed in our arguments is that
if $K$ is an algebraically closed field of characteristic 0, then there is an equivalence of categories between the category of connected algebraic rings over $K$ (i.e. those algebraic rings whose underlying variety is irreducible) and the category of finite-dimensional $K$-algebras. Yet another ingredient in the proof is the analysis of the group $K_2$ of an algebraic ring, for which we rely on Stein's \cite{St2} description of $K_2$ of a semilocal commutative ring.

Our methods also allow us to analyze abstract representations of elementary subgroups of groups of type $A_n$ over associative, but not necessarily commutative, rings. In this case, however, no direct analogues of Stein's results were immediately available, so we developed the necessary $K$-theoretic machinery for the noncommutative case using computations by Bak and Rehmann \cite{BR} of relative $K_2$ groups.
For simplicity, we only give the statement of one of our results concerning (BT) here, and refer the reader to \S \ref{S:NC} for the most general statements. First, recall that
if $D$ is a finite-dimensional central division algebra over a field $k$, then $G = \mathbf{SL}_{n,D}$ is the algebraic $k$-group such that $G(k) = SL_n (D),$ the group of elements of $GL_n (D)$ having reduced norm one (it is well-known that $G$ is an inner form of type $A_{\ell}$ -- cf. \cite{KMRT} or \cite{PR} for the details).

\vskip2mm

\noindent {\bf Theorem 2.} (\cite{IR1}, Theorem 1) {\it Let $D$ be a finite-dimensional central division algebra over a field $k$ of characteristic 0, and let $G = \mathbf{SL}_{n,D}$, where $n \geq 3.$ Let $\rho \colon G(k) \to GL_m (K)$ be a finite-dimensional linear representation of $G(k)$ over an algebraically closed field $K$ of characteristic 0, and set $H = \overline{\rho (G(k))}$ (Zariski-closure). Then the abstract homomorphism $\rho \colon G(k) \to H(K)$ has a standard description.}

\vskip2mm

Now, one motivation for studying
representations with non-reductive image comes from the fact that such representations arise naturally in the analysis of {\it deformations of representations}. Formally, over a field of characteristic 0, deformations of (completely reducible) $n$-dimensional representations of a finitely generated group $\Gamma$ can be understood in terms of the corresponding character variety $X_n (\Gamma).$
We have used Theorem 1 to estimate the dimension of the character variety $X_n (\Gamma)$ as a function of $n$ in the case that $\Gamma$ is an elementary subgroup of a Chevalley group over a finitely generated commutative ring.
Before formulating our result, let us first recall the relevant definitions.

Suppose $\Gamma$ is a finitely generated group and $K$ is an algebraically closed field of characteristic 0. For a fixed integer $n \geq 1$, the set $R_n (\Gamma)$ of all $n$-dimensional linear representations $\rho \colon \Gamma \to GL_n (K)$ is an algebraic variety over $K$ called the $n$-{\it th representation variety}. Furthermore, there is a natural conjugation action of $GL_n (K)$ on $R_n (\Gamma)$, and the corresponding quotient variety $X_n (\Gamma)$ is called the $n$-{\it th character variety} (concretely, the points of $X_n (\Gamma)$ correspond to the isomorphism classes of semisimple $n$-dimensional representations of $\Gamma$ -- see \cite{LM}, Chapter 1, for the details). Now,
if $\Phi$ is a reduced irreducible root system of rank $\geq 2$, $G$ the universal Chevalley-Demazure group scheme of type $\Phi$, and $R$ a finitely generated commutative ring, then it has been shown by Ershov, Jaikin, and Kassabov \cite{EJK} that the elementary subgroup $G(R)^+ \subset G(R)$ has Kazhdan's property (T); in particular, it is a  finitely generated group, so its representation and character varieties are defined. Our result is as follows.

\vskip2mm

\noindent {\bf Theorem 3.} (\cite{IR1}, Theorem 2) {\it Let $\Phi$ be a reduced irreducible root system of rank $\geq 2$, $R$ a finitely generated commutative ring such that $(\Phi, R)$ is a nice pair, and $G$ the universal Chevalley-Demazure group scheme of type $\Phi$.
Denote by $\Gamma$ the elementary subgroup $G(R)^+$ of $G(R)$ and consider $n$-th character variety
$X_n (\Gamma)$ of $\Gamma$ over an algebraically closed field $K$ of characteristic 0. Then there exists a constant $c = c(R)$ (depending only on $R$) such that $\varkappa_n (\Gamma) := \dim X_n (\Gamma)$ satisfies
$$
\varkappa_n (\Gamma) \leq c \cdot n
$$
for all $n \geq 1.$}

\vskip2mm

To put Theorem 3 into perspective, recall that for $\Gamma = F_d,$ the free group on $d > 1$ generators,
we have
$$
\varkappa_n (\Gamma) = (d-1)n^2 + 1,
$$
i.e. the growth of $\varkappa_n (\Gamma)$ is {\it quadratic} in $n.$ It follows that the rate of growth cannot be more than quadratic for {\it any} finitely generated group (and it is indeed quadratic in some important situations, such as $\Gamma = \pi_g$, the fundamental group of a compact orientable surface of genus $g > 1$, cf. \cite{RBC}). At the other end of the spectrum are the groups $\Gamma$, called $SS$-{\it rigid}, for which $\varkappa_n (\Gamma) = 0$ for all $n \geq 1$. For example, according to Margulis's Superrigidity Theorem (\cite{Mar}, Ch. VII, Theorem 5.6 and 5.25, Theorem A), all irreducible higher-rank lattices are $SS$-rigid. Now, one can show using (\cite{AR1}, \S 2) that if $\Gamma$ is not $SS$-rigid, then the rate of growth of $\varkappa_n (\Gamma)$ is at least linear. It follows that unless $\Gamma$ is $SS$-rigid, the growth rate of $\varkappa_n (\Gamma)$ is between linear and quadratic, and Theorem 3 shows that it is the minimal possible for $\Gamma = G(R)^+.$

The proof of Theorem 3 is
based on a suitable variation of the approach, going back to A.~Weil \cite{W}, of bounding the dimension of the tangent space to $X_n (\Gamma)$ at a point $[\rho]$ corresponding to a representation $\rho \colon \Gamma \to GL_n (K)$ by the dimension of the cohomology group $H^1 (\Gamma, \mathrm{Ad}_{GL_n} \circ \rho).$ Using Theorem 1,
we describe the latter space in terms of certain spaces of derivations of $R.$  This leads to the conclusion that the constant $c$ in Theorem 3 does not exceed the minimal number of generators $d$ of $R$ (i.e. the smallest integer such that there exists a surjection $\Z[X_1, \dots, X_d] \twoheadrightarrow R$). In fact, if $R$ is the ring of integers or $S$-integers in a number field $L$, then $c = 0$, so we obtain that
$\varkappa_n (\Gamma) = 0$ for all $n$, i.e. $\Gamma$ is $SS$-rigid. We should add that
Theorem 1 can be used to show that $\Gamma = G(R)^+$ is actually {\it superrigid} in this case.

Now, let us observe that groups of the form $G(R)^+$ account
for most of the known examples of linear Kazhdan groups that are not lattices, so it is natural to ask if the conclusion of Theorem 3 can be extended to \emph{all} discrete linear Kazhdan groups (note that the result of Ershov, Jaikin, and Kassabov \cite{EJK} mentioned above yields a large number of examples of Kazhdan groups that are \emph{not} $SS$-rigid)


\vskip2mm

\noindent {\bf Conjecture.} {\it Let $\Gamma$ be a discrete linear group having Kazhdan's property (T). Then there exists a constant $c = c(\Gamma)$ such that $$\varkappa_n (\Gamma) \leq c \cdot n$$ for all $n \geq 1.$}

\vskip3mm

We should point out that if the assumption that $\Gamma$ is a linear is dropped, then one can produce counterexamples to the conjecture (see \cite{EJK}).

In summary, our approach to rigidity questions for Chevalley groups has two aspects:
first, we analyze individual abstract representations of elementary subgroups of Chevalley groups and obtain standard descriptions in many situations;
then we apply these results to study character varieties of elementary subgroups, thereby obtaining information about the totality of {\it all} representations.


\vskip2mm

The paper is organized as follows. We begin by giving an overview in \S \ref{S:AR} of the basic algebraic and geometric properties of algebraic rings that are needed for the proofs of our main results. Next,
in \S \ref{S:KT}, we summarize some facts about Steinberg groups and $K_2$ that are used in the proof of Theorem 1, and also indicate how to obtain the necessary analogues of Stein's results in the noncommutative setting. In \S \ref{S:T1, T2}, we indicate the main steps in the proof of Theorem 1, and then also
give some remarks regarding the proof of Theorem 2. Finally, we sketch the proof of Theorem 3 in \S \ref{S:T-3}.

\vskip2mm

\noindent {\bf Notations and conventions.} Throughout this paper, $\Phi$ will denote a reduced irreducible root system of rank $\geq 2.$ All of our rings are unital and associative.
We let $\mathbb{G}_a = {\rm Spec} \ \Z[T]$ and $\mathbb{G}_m = {\rm Spec} \ \Z[T, T^{-1}]$ be the standard additive and multiplicative group schemes over $\Z$, respectively. Also, as noted earlier, if $R$ is a commutative ring, we say that the pair $(\Phi, R)$ is {\it nice} if $2 \in R^{\times}$ whenever $\Phi$ contains a subsystem of type $B_2$, and $2, 3 \in R^{\times}$ if $\Phi$ is of type $G_2.$ Finally, given an algebraic group $H$ (resp., an algebraic ring $A$), we let $H^{\circ}$ (resp., $A^{\circ}$) denote the connected component of the identity (resp., of zero).

\vskip2mm

\noindent {\bf Acknowledgements.} I would like to thank Shing-Tung Yau and Lizhen Ji for the opportunity to speak at the Conference on Group Actions and Applications in Geometry, Topology, and Analysis during a wonderful visit to Kunming. The results discussed in this survey paper are part of my PhD thesis, and I would like to thank my advisor Professor Gregory A. Margulis for his guidance over the last several years. I would like to thank Martin Kassabov for his interest in my work and for useful discussions. Finally, I would like to thank Brian Conrad and Gopal Prasad for informative discussions about pseudo-reductive groups.

\vskip5mm

\section{On algebraic rings}\label{S:AR}

In this section, we present a survey of some basic properties of algebraic rings. This material originally appeared in our paper \cite{IR}, where our motivation was to extend a construction of Kassabov and Sapir \cite{Kas}
in order to obtain rigidity statements over general commutative rings; later, it was pointed out to us by Kassabov that some, but not all, of these results could be found in M.~Greenberg's paper \cite{Gr}. In our exposition, we follow \cite[\S 2]{IR}, to which we refer the reader for full details.
All algebraic varieties considered in this section will be over a fixed algebraically closed field $K$.

\vskip5mm

\subsection{Definitions and first properties} An algebraic ring is essentially an algebraic variety with a ring structure defined by regular maps. More formally, we have the following definition.

\vskip2mm

\noindent {\bf Definition 2.1.1.} An {\it algebraic ring} is a triple $(A, \ba, \bm)$ consisting of an {\it affine} algebraic variety $A$ and two regular maps $\ba \colon A \times A \to A$ and $\bm \colon A \times A \to A$ (``addition" and ``multiplication") such that

\vskip1mm

\noindent (I) \parbox[t]{15cm}{$(A, \ba)$ is a commutative algebraic group (in particular, there exists an element $0_A \in A$ such that $\ba (x, 0_A)= \ba (0_A, x) = x$ and there is a regular map $\iota \colon A \to A$ such that $\ba (x, \iota (x)) = \ba (\iota(x), x) = 0_A$, for all $x \in A$);}

\vskip1mm

\noindent (II) \parbox[t]{15cm}{$\bm (\bm(x,y), z) = \bm (x, \bm (y,z))$, for all $x, y, z \in A$ (``associativity");}

\vskip1mm

\noindent (III) \parbox[t]{15cm}{$\bm (x, \ba (y,z)) = \ba (\bm(x,y), \bm (x, z))$ and $\bm (\ba (x,y), z) = \ba (\bm (x,z), \bm (y,z))$, for all $x, y, z \in A$ (``distributivity").}

\vskip2mm
\noindent An algebraic ring $(A, \ba, \bm)$ is called {\it commutative} if $\bm (x,y) = \bm (y,x)$, for all $x,y \in A$.

\vskip1mm

\noindent The triple $(A, \ba, \bm)$ is an {\it algebraic ring with identity} if in addition to (I)-(III), we have

\vskip1mm

\noindent (IV) \parbox[t]{15cm}{there exists an element $1_A \in A$ such that $ \bm  ( 1_A, x) = \bm (x, 1_A) = x$, for all $x \in A.$}

\vskip5mm

\noindent As a matter of convention, all algebraic rings considered in this paper will be assumed to have an identity element.

\vskip2mm

\noindent We will write $x+y$ and $xy$ for $\ba (x,y)$ and $\bm (x,y)$, respectively, whenever this does not lead to confusion.

\vskip1mm

\noindent {\bf Remark 2.1.2.} Using the Rigidity Lemma from \cite[Chapter 2]{Mum} and Chevalley's result on the structure of smooth connected algebraic groups (\cite{Chev1} or \cite{Con}), one can show that if $A$ is an arbitrary irreducible variety equipped with two regular maps $\ba \colon A \times A \to A$ and $\bm \colon A \times A \to A$ satisfying conditions (I)-(IV), then $A$ is necessarily affine (see \cite[Theorem 2.21]{IR}).


\vskip2mm

Next, morphisms of algebraic rings are defined in the natural way.

\vskip1mm

\noindent {\bf Definition 2.1.3.} Let $(A, \ba, \bm)$ and $(A', \ba', \bm')$ be algebraic rings. A regular map $\varphi \colon A \to A'$ is called a {\it homomorphism of algebraic rings} if
$$
\varphi (\ba(x,y)) = \ba' (\varphi (x), \varphi(y)) \ \ \ {\rm and} \ \ \ \varphi (\bm   (x,y)) = \bm' (\varphi(x), \varphi(y)).
$$
If $A$ and $A'$ are rings with identity, we also require that
$$
\varphi (1_A) = 1_{A'}.
$$
If $\varphi$ is in addition an isomorphism of algebraic varieties, then $\varphi$ is called an {\it isomorphism of algebraic rings}.

\vskip2mm

\subsection{Units of algebraic rings} We begin by considering the group of units $A^{\times}$ of an algebraic ring $A$.



\begin{prop}\label{P:AR-1}
Let $A$ be an algebraic ring. Then

\vskip1mm

\noindent {\rm (i)} \parbox[t]{14cm}{The group of units $A^{\times}$ is a (nonempty) principal open subset of $A$.}

\vskip1mm

\noindent {\rm (ii)} \parbox[t]{14cm}{The map $A^{\times} \to A^{\times}, t \mapsto t^{-1}$ is regular. In particular, $(A^{\times}, \bm)$ is an algebraic group.}
\end{prop}

The proposition immediately yields the following corollary, which will be essential in our applications of Stein's \cite{St2} results on $K_2$.
\begin{cor}\label{C:AR-1}
Let $A$ be a connected algebraic ring (i.e., the underlying variety of $A$ is irreducible). Then $A = A^{\times} - A^{\times}$; in other words, $A$ is generated by its units.
\end{cor}

We also note that Proposition \ref{P:AR-1} puts strong restrictions on algebraic division rings (i.e., algebraic rings $A$ in which $A^{\times} = A \setminus \{ 0 \}$).
\begin{cor}\label{C:AR-3}
Let $A$ be an algebraic division ring. Then $\dim A \leq 1.$
\end{cor}
\begin{proof}
By Proposition \ref{P:AR-1}, $A^{\times}$ is a principal open subset, say $A^{\times} = D(\chi).$ The closed set $V(\chi) := A \setminus D(\chi)$ contains 0, hence is nonempty. By the Dimension Theorem, $\dim V(\chi) \geq \dim A -1.$ On the other hand, $V(\chi) = \{ 0 \},$ so $\dim A \leq 1.$
\end{proof}
Now, the algebraic rings that arise in the proof of Theorems 1
come equipped with a homomorphism of abstract rings $f \colon R \to A$, where $R$ is a given commutative abstract ring, such that $\overline{f(R)} = A$ (the bar denotes Zariski closure); consequently, we will need
results that relate the algebraic properties of $R$ and $A$. One such statement is the following.

\begin{cor}\label{C:AR-2}
Let $f \colon R \to A$ be an abstract ring homomorphism of an abstract commutative semilocal ring into a connected commutative algebraic ring such that $\overline{f(R)} = A.$ Then $\overline{f(R^{\times})} = A.$
\end{cor}
\begin{proof}
Let $\mathfrak{m}_1, \dots, \mathfrak{m}_{\ell}$ be the maximal ideals of $R$ such that $\overline{f(\mathfrak{m}_i)} = A,$ and let $\mathfrak{m}_{\ell+1}, \dots, \mathfrak{m}_n$ be all other maximal ideals. It follows from Proposition \ref{P:AR-1} that $f(\mathfrak{m}_i) \cap A^{\times} \neq \varnothing$ for all $i = 1, \dots, \ell,$ and therefore $f(\mathfrak{m}_1 \cdots \mathfrak{m}_{\ell}) \cap A^{\times} \neq \varnothing.$ The assumption $\overline{f(R)} = A$ implies that $\overline{f (\mathfrak{m}_1 \cdots \mathfrak{m}_{\ell})}$ is an ideal of $A$, so we conclude that $\overline{f(\mathfrak{m}_1 \cdots \mathfrak{m}_{\ell})} = A.$ Now the fact that $A$ is connected implies that $V := A \setminus \cup_{i = \ell + 1}^n \overline{f(\mathfrak{m}_i)}$ is a dense open subset of $A$, and therefore $T:= f(1 + \mathfrak{m}_1 \cdots \mathfrak{m}_{\ell}) \cap V$ is dense in $A$. On the other hand, if $x \in 1 + \mathfrak{m}_1 \cdots \mathfrak{m}_{\ell}$ is such that $f(x) \in V,$ then $x \not\in \cup_{i = 1}^n \mathfrak{m}_i,$ and therefore $x \in R^{\times}.$ Thus, $T \subset f(R^{\times}),$ so the latter is dense in $A$.
\end{proof}

Next, our description of $A^{\times}$ enables us to show that any commutative algebraic ring $A$ is automatically semilocal as an abstract ring. First, let us note that if $\mathfrak{a}$ is a closed 2-sided ideal of an algebraic ring $A$, then the quotient $A/\mathfrak{a}$ in the category of additive algebraic groups is easily seen to have a natural structure of an algebraic ring.

\begin{lemma}\label{L:AR-1}
Let $A$ be a commutative algebraic ring. Then $A$ is semilocal as an abstract ring.
\end{lemma}
\begin{proof}
First, we observe that any abstract maximal ideal $\mathfrak{m} \subset A$ is Zariski-closed. Indeed, if $\mathfrak{m}$ is not closed, then since $\overline{\mathfrak{m}}$ is an ideal, we have $\overline{\mathfrak{m}} = A.$ Then it follows from Proposition \ref{P:AR-1} that $\mathfrak{m} \cap A^{\times} \neq \varnothing$, which is impossible.

Now let $\mathfrak{m}_1, \dots, \mathfrak{m}_n$ be a collection of pairwise distinct maximal ideals of $A$. It follows from the Chinese Remainder Theorem that the canonical map
$$
A \stackrel{\nu}{\rightarrow} A/ \mathfrak{m}_1 \times \cdots A/ \mathfrak{m}_n
$$
is a surjective homomorphism of algebraic rings. If $\dim A/ \mathfrak{m}_i = 0,$ then $A/ \mathfrak{m}_i$ is finite, and consequently $\mathfrak{m}_i \supset A^{\circ}.$ Otherwise, $\dim A/ \mathfrak{m}_i \geq 1.$ The surjectivity of $\nu$ now implies that $n \leq r + \dim A$, where $r$ is the number of maximal ideals of the finite ring $A/ A^{\circ},$ and the required fact follows.
\end{proof}

\vskip5mm

\subsection{Structure of algebraic rings} We will now consider more closely the algebraic structure of algebraic rings. The key statement, given in Proposition \ref{P:AR-2} below, is a complete description of algebraic rings over fields of characteristic 0. However, some important structural information about algebraic rings (particularly
commutative algebraic rings) can be obtained in any characteristic. We begin with the artinian property.
\begin{lemma}\label{L:AR-2}
Let $A$ be a connected algebraic ring. Then every right and every left ideal of $A$ is connected and Zariski-closed, hence $A$ is artinian as an abstract ring.
\end{lemma}
\begin{proof}
Notice that the second statement follows from the first as $A$ is a noetherian topological space for the Zariski topology. So, suppose $\mathfrak{a} \subset A$ is a left ideal. Then for any $a_1, \dots, a_n \in \mathfrak{a}$, the left ideal $\mathfrak{b} \subset A$ generated by $a_1, \dots, a_n$ is the image of the homomorphism of algebraic groups $A^n \to A$, $(x_1, \dots, x_n) \mapsto x_1 a_1 + \dots + x_n a_n,$ and therefore is closed (\cite{Bo}, Corollary 1.4) and connected. If $a_1, \dots, a_n \in \mathfrak{a}$ are chosen so that the corresponding $\mathfrak{b}$ has maximum possible dimension, then clearly $\mathfrak{a} = \mathfrak{b}.$ The argument for right ideals is completely analogous. \end{proof}

\vskip2mm

\noindent {\bf Remark 2.3.2.} We note that without the assumption of connectedness, the conclusion of the lemma may fail. For example, suppose $K$ is an algebraically closed field of characteristic $p > 0.$ Let $A_0 = K \times K,$ with the usual addition and multiplication given by
$$
\bm((x_1, y_1), (x_2, y_2)) = (x_1 y_2 + x_2 y_1, y_1 y_2).
$$
It is easily seen that $A_0$ is a commutative algebraic ring with identity element $(0,1).$ Then $A = K \times \mathbb{F}_p,$ where $\mathbb{F}_p$ is the prime subfield of $K$, is an algebraic subring of $A_0.$ Now if $S \subset K$ is any additive subgroup, then $\mathfrak{a} = (S,0)$ is an abstract ideal of $A$. Hence $A$ is not artinian and not every ideal of $A$ is Zariski-closed.


\vskip2mm

\addtocounter{thm}{1}

Nevertheless, the artinian property does hold for commutative algebraic rings satisfying one additional condition, which we will now define. Let $A$ be a commutative algebraic ring. The connected component $A^{\circ}$ of $0_A$ in $(A, \ba)$ is easily seen to be an ideal of $A$. Consider the following condition on $A$:
\vskip3mm

\noindent (FG)\ \  \parbox[t]{15cm}{$A^{\circ}$ is finitely generated as an ideal of $A$.}

\vskip3mm
We note that (FG) holds automatically for algebraic rings over fields of characteristic 0 --- see Remark 2.3.7(ii) below.
It turns out that commutative algebraic rings satisfying (FG) possess a number of important structural properties.
\begin{prop}\label{P:AR-4}
Let $A$ be a commutative algebraic ring satisfying {\rm (FG)}. Then
\vskip1mm
\noindent {\rm (i)} \parbox[t]{15cm}{Every abstract ideal ideal $\mathfrak{a} \subset A$ is Zariski-closed, and consequently $A$ is artinian.}
\vskip1mm
\noindent {\rm (ii)} \parbox[t]{15cm}{We have the following direct sum decomposition of algebraic rings with identity
$$
A = A^{\circ} \oplus C,$$
where $C$ is a finite ring isomorphic to $A/ A^{\circ}$.}
\end{prop}

\vskip2mm

An important class of examples of algebraic rings satisfying condition (FG) is obtained as follows.
\begin{lemma}\label{L:AR-3}
Let $f \colon R \to A$ be an abstract homomorphism of an abstract commutative ring $R$ into a commutative algebraic ring $A$ such that $\overline{f(R)} = A.$
\vskip1mm

\noindent {\rm (i)} \parbox[t]{15cm}{If $R$ is noetherian, then $A$ satisfies {\rm (FG)}.}

\vskip1mm

\noindent {\rm (ii)} \parbox[t]{15cm}{If $R$ is an infinite field, then $A$ is connected.}

\end{lemma}
\begin{proof}
(i) Since $A^{\circ}$ is open, for $\mathfrak{r} = f^{-1} (f(R) \cap A^{\circ})$, we have $\overline{f(\mathfrak{r})} = A^{\circ}.$ By assumption, $R$ is noetherian and clearly $\mathfrak{r}$ is an ideal, so we have $\mathfrak{r} = R u_1 + \dots + R u_m.$ Then
since $A f(u_1) + \cdots + A f(u_m)$ is closed and contains $f(\mathfrak{r}),$ we obtain $A^{\circ} = A f(u_1) + \cdots + A f(u_m).$

\vskip1mm
\noindent (ii) If $A \neq A^{\circ},$ then $f^{-1}(A^{\circ})$ would be a proper ideal of $R$ of finite index, which is impossible.
\end{proof}

\vskip2mm

\noindent {\bf Remark 2.3.5.} Notice that if $D$ is a division algebra over an infinite field, $A$ an associative algebraic ring, and $f \colon D \to A$ a ring homomorphism with Zariski-dense image, then the same argument as given in (ii) shows that $A$ must be connected.

\vskip2mm

\addtocounter{thm}{1}

Suppose $S$ is a finite-dimensional $K$-algebra. Then $S$ has a natural structure of a connected algebraic ring. We will say that an algebraic ring $A$ {\it comes from an algebra} if there exists a finite dimensional $K$-algebra $S$ and an isomorphism $A \simeq S$ of algebraic rings. Furthermore, we say that an algebraic ring $A$ {\it virtually comes from an algebra} if $A \simeq A' \oplus B$, where
$A'$ comes from an algebra and $B$ is finite. For such algebraic rings most of our previous results are immediate.
On the other hand, it turns out that in characteristic zero, all algebraic rings virtually come from algebras.
\begin{prop}\label{P:AR-2}
Let $A$ be an algebraic ring over an algebraically closed field $K$ of characteristic zero. Then $A$ virtually comes from an algebra.
\end{prop}
\begin{proof}(Sketch) One first shows that if $A$ is an algebraic ring over an algebraically closed field $K$ of \emph{any} characteristic, then we have $A = A' \oplus C$, where $A'$ is an algebraic subring of $A$ consisting of all unipotent elements in $(A, \ba)$, and $C$ is a finite ring consisting of all semisimple elements. Then, using the fact that commutative unipotent groups over fields of characteristic 0 are vector groups (see \cite[7.3]{Bo}), it is easy to show that $A'$ is actually a finite-dimensional $K$-algebra.
\end{proof}

\vskip2mm

\noindent {\bf Remark 2.3.7.} (i) The proof of Proposition \ref{P:AR-2} sketched above (see \cite[Proposition 2.14]{IR} for the details) is purely algebraic and works for \emph{any} algebraically closed field of characteristic 0. In the case that $K = \C$, Kassabov and Sapir \cite{Kas} proved this result using a topological argument.

\vskip1mm

\noindent (ii) Notice that if $\mathrm{char} \: K = 0$, the subring $A'$ coincides with $A^{\circ}$ and is generated as an ideal of $A$ by $1_{A'}$, the projection of $1_A$ to $A'$; hence, $A$ satisfies condition (FG) introduced above.

\addtocounter{thm}{1}

\vskip5mm

\subsection{Commutative artinian algebraic rings} To conclude our discussion of algebraic rings, we would like to mention some structural results for commutative artinian algebraic rings that hold regardless of whether
or not the ring (virtually) comes from an algebra. So, let $A$ be a commutative artinian algebraic ring with identity,
${\mathfrak m}_1, \ldots , {\mathfrak m}_r$ be its maximal ideals (note that $A$ is semilocal by \cite{At}, Proposition 8.3), and $J = {\mathfrak m}_1 \cap \cdots \cap {\mathfrak m}_r$
be its Jacobson radical (recall that the ${\mathfrak m}_i$'s, hence also $J,$ are Zariski-closed by Lemma \ref{L:AR-1}). It is well-known (cf. \cite{At}, Theorem 8.7) that
there are idempotents $e_1, \ldots , e_r \in A$ such that $e_1 + \cdots + e_r = 1_A,$ $e_ie_j = 0$ for $i \neq j,$ and $A_i := e_iA$ is a {\it
local} commutative artinian ring with identity for each $i = 1, \ldots , r.$ We have
\begin{equation}\label{E:AR120}
A \simeq A_1 \oplus \cdots \oplus A_r, \  \  \  a \mapsto (e_1a, \ldots , e_ra),
\end{equation}
as algebraic rings. Furthermore, after possible renumbering, the ideal ${\mathfrak m}_i$ corresponds to $$A_1 \oplus \cdots \oplus {\mathfrak m}'_i
\oplus \cdots A_r,$$ where ${\mathfrak m}'_i$ is the unique maximal ideal of $A_i.$

Let now $B$ be a local commutative artinian algebraic ring with maximal ideal ${\mathfrak n}.$ It follows
from Corollary \ref{C:AR-3} that $\dim B/{\mathfrak n} \leq 1.$ If $\dim B/{\mathfrak n} = 0,$ i.e. $B/{\mathfrak n}$ is finite, then since
there exists $n \geq 1$ such that ${\mathfrak n}^n = \{ 0 \}$ (\cite{At}, Proposition 8.4) and each quotient ${\mathfrak n}^j/{\mathfrak n}^{j+1}$ in the
filtration
\begin{equation}\label{E:AR121}
B \supset {\mathfrak n} \supset {\mathfrak n}^2 \supset \cdots \supset {\mathfrak n}^{n-1} \supset {\mathfrak n}^n = \{ 0 \}
\end{equation}
is a finitely generated $B/{\mathfrak n}$-module (as it is an artinian module over the field $B/ \mathfrak{n}$), we conclude that $B$ itself is finite. Now, suppose $\dim B/{\mathfrak n} = 1.$
Since $B/{\mathfrak n}$ is an infinite algebraic division ring, it is automatically connected, so we can use the following.

\begin{prop}\label{P:AR-3}
Suppose $(A, \ba, \bm)$ is a one-dimensional connected commutative algebraic ring. Then $(A, \ba , \bm) \simeq (K, +, \cdot)$ as algebraic rings.
\end{prop}

Thus, if $\dim  B/{\mathfrak n} = 1$, then $B/{\mathfrak n} \simeq K$ with the natural operations. Combining this with the decomposition (\ref{E:AR120}) and taking into account that
$J = {\mathfrak m}'_1 \oplus \cdots \oplus {\mathfrak m}'_r,$ we obtain the following
\begin{prop}\label{P:AR-20}
Let $A$ be a commutative artinian algebraic ring with Jacobson radical $J.$ Then

\vskip1mm

\noindent {\rm (i)} $A \simeq A_1 \oplus \cdots \oplus A_r$, where each $A_i$ is a local commutative artinian
algebraic ring;

\vskip1mm

\noindent {\rm (ii)} \parbox[t]{15cm}{$A/{J} \simeq (K, +, \cdot)^n \oplus C$ where, $n = \dim A/J$ and $C$ is
a finite algebraic ring; in particular, $A/J$ always virtually comes from an algebra.}

\end{prop}

\vskip5mm

\section{$K$-theoretic background}\label{S:KT}
In this section, we describe the $K$-theoretic machinery that is needed in the proofs of Theorems 1 and 2. The cases of commutative and noncommutative rings are treated separately (in \S\S \ref{S:KT} and \ref{S:KTNC}, respectively). In the former, the necessary statements are given by results of Stein \cite{St2} on the group $K_2$ of a semilocal commutative ring.
In the latter, we sketch how analogous results in the relevant situations can be obtained using some computations of Bak and Rehmann \cite{BR}.

\vskip2mm

\subsection{Structure of $K_2$: commutative case}\label{S:KT} We begin by recalling some standard definitions.
Let $\Phi$ be a reduced irreducible root system of rank $\geq 2$, $G$ be the corresponding universal Chevalley group scheme, and $e_{\alpha} \colon \mathbb{G}_a \to G$ be the one-parameter subgroup associated with $\alpha \in \Phi.$ Suppose $S$ is a commutative ring. It is well-known (see \cite{Stb1}, Chapter 3) that the elements $e_{\alpha} (s)$, for $\alpha \in \Phi$ and $s \in S$, satisfy the following relations:
\begin{equation}\label{E:StG103}
e_{\alpha} (s) e_{\alpha} (t) = e_{\alpha} (s+t)
\end{equation}
for all $s,t \in S$ and all $\alpha \in \Phi,$ and
\begin{equation}\label{E:StG104}
[e_{\alpha} (s), e_{\beta} (t)] = \prod e_{i \alpha + j \beta} (N^{i,j}_{\alpha, \beta} s^i t^j),
\end{equation}
for all $s,t \in S$ and all $\alpha, \beta \in \Phi,$ $\beta \neq - \alpha,$
where the product is taken over all roots of the form $i \alpha + j \beta,$ $i, j \in \Z^+$, listed in an arbitrary (but {\it fixed}) order, and the $N^{i,j}_{\alpha, \beta}$ are integers depending only on  $\Phi$ and the order of the factors in (\ref{E:StG104}), but not on the ring $S$.

\vskip2mm

\noindent {\bf Definition 3.1.1.} The {\it Steinberg group} $\mathrm{St}(\Phi, S)$ is the group generated by symbols
$\tilde{x}_{\alpha} (t)$, for all $t \in S$ and $\alpha \in \Phi$, subject to the relations
\vskip1mm

(R1) $\tilde{x}_{\alpha}(s) \tilde{x}_{\alpha}(t) = \tilde{x}_{\alpha} (s+t)$

\vskip1mm

(R2) $[\tilde{x}_{\alpha} (s), \tilde{x}_{\beta} (t)] = \prod \tilde{x}_{i \alpha + j \beta} (N^{i,j}_{\alpha, \beta} s^i t^j)$,

\vskip1mm

\noindent where $N^{i,j}_{\alpha, \beta}$ are the same integers as in (\ref{E:StG104}).

\vskip2mm
\addtocounter{thm}{1}

It follows from the definition and relations (\ref{E:StG103}) and (\ref{E:StG104}) that there exists a surjective group homomorphism
$$
\pi_S \colon \mathrm{St}(\Phi, S) \to G(S)^+, \ \ \ \tilde{x}_{\alpha} (t) \mapsto e_{\alpha} (t),
$$
and one defines
$$
K_2 (\Phi, S) = \ker \pi_S.
$$
We note that the pair $(\St (\Phi, S), \pi_S)$ is functorial in the following sense: given a homomorphism of commutative rings
$f \colon S \to T$, there is a commutative diagram of group homomorphisms
$$
\xymatrix{ \St (\Phi, S) \ar[d]_{\pi_S} \ar[r]^{\tilde{F}} & \St (\Phi, T) \ar[d]^{\pi_T} \\ G(S)^+ \ar[r]^{F} & G(T)^+}
$$
where $F$ and $\tilde{F}$ are the natural homomorphisms induced by $f$, i.e. that map the generators as follows:
$F \colon e_{\alpha} (t) \mapsto e_{\alpha} (f(t))$ and $\tilde{F} \colon \tilde{x}_{\alpha} (t) \mapsto \tilde{x}_{\alpha} (f(t)).$

Next, we recall some additional standard notations.
For $\alpha \in \Phi$ and $u \in S^{\times}$, define the following elements of $\St (\Phi, S):$
\begin{equation}\label{E:StG-20}
\tilde{w}_{\alpha} (u) = \tilde{x}_{\alpha} (u) \tilde{x}_{-\alpha} (-u^{-1}) \tilde{x}_{\alpha} (u) \ \ \ {\rm and} \ \ \  \tilde{h}_{\alpha}(u) = \tilde{w}_{\alpha} (u) \tilde{w}_{\alpha} (-1).
\end{equation}
Then,
for $u, v \in S^{\times}$, the {\it Steinberg symbol} $\{ u, v \}_{\alpha}$ is defined as
\begin{equation}\label{E-Symb1}
\{ u, v \}_{\alpha} = \tilde{h}_{\alpha} (uv) \tilde{h}_{\alpha} (u)^{-1} \tilde{h}_{\alpha} (v)^{-1}.
\end{equation}
We note that since the elements $h_{\alpha} (t) = \pi (\tilde{h}_{\alpha} (t))$ are multiplicative in $t$ (\cite{Stb1}, Lemma 28), all Steinberg symbols are contained in $K_2 (\Phi, S).$ Moreover, by
(\cite{St2}, Proposition 1.3(a)), the Steinberg symbols lie in the center of $\St (\Phi, S).$

Historically, central extensions of the groups of points of Chevalley groups over fields were first studied by Steinberg \cite{Stb3}. He showed that for a field $k$ with $\vert k \vert > 4$, $\pi_k \colon \St (\Phi, k) \to G(k)^+$ is a {\it universal central extension}. The structure of $\ker \pi_k$ was completely described by Matsumoto \cite{M1}: he proved that $\ker \pi_k$ is the subgroup of $\St (\Phi, k)$ generated by Steinberg symbols (\ref{E-Symb1}) and determined the full set of relations (we should point out that Matsumoto, whose main goal was the resolution of the congruence subgroup problem for semisimple split groups over number fields, in fact only considered the case of infinite topological fields, but, as noted by Steinberg in \cite[Theorem 12]{Stb1}, the argument actually works for all fields $k$ with $\vert k \vert > 4$). Motivated by this work, Milnor formally introduced the functor $K_2$ from the category of associative rings to the category of abelian groups for groups of type $A_n$ (a detailed exposition can be found in his book \cite{Mil}). These ideas then played a crucial role in the work of Bass, Milnor, and Serre \cite{BMS} on the congruence subgroup problem.

Several years later, Stein \cite{St2} established an analogue of Matsumoto's result for commutative semilocal rings. To give the statement, we will need the following notation: for any ring $S$ and any nonempty subset  $P \subset S$, we denote by $\Z [P]$ the subring of $S$ generated by $P.$

\begin{thm}\label{T:StG-1}{\rm (cf. \cite{St2}, Theorem 2.13)}
Let $\Phi$ be a reduced irreducible root system of rank $\geq 2.$ If $S$ is a commutative semilocal ring with $S = \Z[S^{\times}]$, then
$K_2 (\Phi, S)$ is the central subgroup of $\St (\Phi, S)$ generated by the Steinberg symbols $\{ u, v \}_{\alpha}$, with $u,v \in S^{\times}$, for {\rm any} fixed long root $\alpha.$
\end{thm}
Our results from $\S$ \ref{S:AR} yield the following.
\begin{cor}\label{C:StG-1}
Let $A$ be a connected commutative algebraic ring. Then $K_2 (\Phi, A)$ is a central subgroup of $\St (\Phi, A)$ generated by the Steinberg symbols $\{ u, v \}_{\alpha}$ for $u, v \in A^{\times}$ and any fixed long root $\alpha.$
\end{cor}
\begin{proof}
Since $A$ is connected, we have $A = A^{\times} - A^{\times}$ (Corollary \ref{C:AR-1}), hence $A = \Z [A^{\times}]$. Furthermore, by Lemma \ref{L:AR-1}, any commutative algebraic ring is semilocal as an abstract ring. Thus, the statement follows directly from Theorem \ref{T:StG-1}.
\end{proof}

The centrality of $K_2 (\Phi, A)$ will be critical for the proof of Theorem 1. We also note the following finiteness statement.
\begin{prop}\label{P:StG-2} {\rm (\cite{IR}, Proposition 4.5)}
Let $S$ be a finite commutative ring and $\Phi$ a reduced irreducible root system of rank $\geq 2.$ Then $\St (\Phi, S)$ and $K_2 (\Phi, S)$ are finite.
\end{prop}

Now applying this to algebraic rings, we obtain

\begin{prop}\label{P:StG-3}
Suppose $A$ is a commutative algebraic ring over an algebraically closed field $K$ and $\Phi$ is a reduced irreducible root system of rank $\geq 2.$ Assume that $A$ satisfies {\rm (FG)} if $\mathrm{char} \: K = p>0$.
Then $\St (\Phi, A) = \St (\Phi, A^{\circ}) \times P,$ where $P$ is a finite group.
\end{prop}
\begin{proof}
Applying Proposition \ref{P:AR-2} if $\mathrm{char} \: K = 0$ and Proposition \ref{P:AR-4} if $\mathrm{char} \: K > 0$ (observe that the latter proposition applies since $A$ is assumed to satisfy (FG)), we conclude that $A = A^{\circ} \times C$, where $C$ is a finite ring.
So, by (\cite{St2}, Lemma 2.12), we have
$$
\St (\Phi, A) = \St (\Phi, A^{\circ}) \times \St (\Phi, C),
$$
and by Proposition \ref{P:StG-2}, $P := \St (\Phi, C)$ is finite.
\end{proof}
\begin{cor}\label{C:StG-2}
Suppose $R$ is an abstract commutative ring and $A$ is an algebraic ring over a field $K$ such that there exists an abstract ring homomorphism $f \colon R \to A$ with Zariski-dense image. Assume moreover that $R$ is noetherian if $\mathrm{char} \: K > 0.$ Then $\St (\Phi, A) = \St (\Phi, A^{\circ}) \times P,$ where $P$ is a finite group.
\end{cor}
\begin{proof}
Note that if $\mathrm{char} \: K > 0$, the fact that $R$ is noetherian implies that $A$ satisfies (FG) (see Lemma \ref{L:AR-3}).
So our assertion follows immediately from Proposition \ref{P:StG-3}.
\end{proof}

\vskip5mm

\subsection{Structure of $K_2$: noncommutative case}\label{S:KTNC}
Many of the constructions of the previous section carry over to groups of type $A_n$ over associative, but not necessarily commutative, rings (details can be found in \cite[\S 2]{IR1}). Let $R$  be an associative unital ring. For $1 \leq i,j \leq n, i \neq j,$ and $r \in R$, let $e_{ij}(r) \in GL_n (R)$ be the elementary matrix with $r$ as the $(i,j)$-th entry, and denote by $E_n (R)$ the subgroup of $GL_n (R)$ generated by all the $e_{ij} (r)$ (this is again called the {\it elementary subgroup}).
If $n \geq 3,$ it is well-known that the elementary matrices in $GL_n (R)$ satisfy the following relations:

\vskip2mm

\noindent (R1) $e_{ij} (r) e_{ij} (s) = e_{ij} (r+s)$;

\vskip2mm

\noindent (R2) $[e_{ij}(r), e_{kl}(s)] =1$ if $i \neq l$, $j \neq k$;

\vskip2mm

\noindent (R3) $[e_{ij} (r), e_{jl} (s)] = e_{il} (rs)$ if $i \neq l.$

\vskip2mm

\noindent Then, just as in the commutative case, the
{\it Steinberg group} over $R$, denoted $\mathrm{St}_n (R)$, is defined to be the group
generated by
all symbols $\tilde{x}_{ij} (r)$, with $1 \leq i,j \leq n$, $i \neq j,$ and $r \in R$, subject to the natural analogues of the relations (R1)-(R3) written in terms of the $\tilde{x}_{ij}(r).$
Again, we have a canonical surjective group homomorphism
$$
\pi_R \colon \mathrm{St}_n (R) \to E_n (R), \ \ \ \tilde{x}_{ij} (r) \mapsto e_{ij} (r),
$$
and we set
$$
K_2 (n, R) = \ker (\mathrm{St}_n (R) \stackrel{\pi_R}{\longrightarrow} E_n (R)).
$$

\vskip2mm

Let us now summarize the results of Bak and Rehmann \cite{BR} dealing with relative $K_2$ groups of associative rings (see Theorem \ref{KT:T-BR1} below). From now on, we will always assume that $n \geq 3.$
As above, let $R$ be an associative unital ring. Given $u \in R^{\times}$, we again define, for $i \neq j$, the following standard elements of $\mathrm{St}_n (R)$:
$$
\tilde{w}_{ij} (u) = \tilde{x}_{ij} (u) \tilde{x}_{ji}(-u^{-1}) \tilde{x}_{ij} (u) \ \ \ \mathrm{and} \ \ \ \tilde{h}_{ij} (u) = \tilde{w}_{ij}(u) \tilde{w}_{ij}(-1).
$$
Notice that the image $\pi_R(\tilde{h}_{ij}(u))$ in $E_n (R)$ is the diagonal matrix with $u$ as the $i$th diagonal entry, $u^{-1}$ as the $j$th diagonal entry, and $1$'s everywhere else on the diagonal. We also define the following noncommutative version of the usual Steinberg symbols: for $u, v \in R^{\times}$, let
$$
c(u,v) = \tilde{h}_{12}(u) \tilde{h}_{12}(v) \tilde{h}_{12}(vu)^{-1}.
$$
One easily sees that $\pi_R (c(u,v))$ is the diagonal matrix with $uvu^{-1} v^{-1}$ as its first diagonal entry and $1$'s everywhere else on the diagonal. Let $U_n (R)$ be the subgroup of $\mathrm{St}_n(R)$ generated by all the $c(u,v)$, with $u,v \in R^{\times}.$

We also need to consider relative versions of these constructions.
Let $\fa$ be a two-sided ideal of $R$ and
$$
GL_n (R, \fa) := \ker (GL_n (R) \to GL_n (R/ \fa))
$$
be the congruence subgroup of level $\fa$.
Define $E_n (R, \fa)$ to be the normal subgroup of $E_n (R)$ generated by all elementary matrices $e_{ij} (a)$, with $a \in \fa.$ Now letting
$$
\mathrm{St}_n (R, \fa) = \ker (\mathrm{St}_n (R) \to \mathrm{St}_n (R / \fa)),
$$
we have a natural homomorphism $\mathrm{St}_n (R, \fa) \to E_n (R, \fa)$,\footnotemark \footnotetext{It can be shown using the presentation of $\mathrm{St}_n (R)$ that $\mathrm{St}_n (R, \fa)$ is the smallest normal subgroup of $\mathrm{St}_n (R)$ containing the elements $\tilde{x}_{ij} (a),$ with $a \in \fa$.} and we set
$$
K_2 (n, R, \fa) = \ker (\mathrm{St}_n (R, \fa) \to E_n (R, \fa)).
$$
Finally, let
$$
U_n (R, \fa) := < c(u, 1+a) \mid u \in R^{\times}, 1+a \in (1 + \fa) \cap R^{\times}>
$$
(notice this is contained in $\mathrm{St}_n (R, \fa)$). We should point out that even though for a noncommutative ring, the groups $U_n (R)$ and $U_n (R, \fa)$ may not lie in $K_2(n,R)$, it is well-known that any element of $K_2(n,R) \cap U_n (R)$ is automatically contained in the center of $\mathrm{St}_n (R)$ (cf. \cite{Mil}, Corollary 9.3).

The following theorem contains a summary of the results of \cite{BR} that will relevant for our purposes.

\begin{thm}\label{KT:T-BR1}{\rm(cf. \cite{BR}, Theorem 2.9 and Corollary 2.11)}
Let $R$ be an associative unital ring. Suppose that $\fa$ is a two sided ideal contained in the Jacobson radical $\mathrm{Rad}(R)$ of $R$, and that $R$ is additively generated by $R^{\times}.$ Assume $n \geq 3.$ Then

\vskip1mm

\noindent {\rm (1)} \parbox[t]{16cm}{$K_2 (n, R, \fa) \subset U_n (R, \fa)$
and the canonical sequence below is exact
$$
1 \to U_n (R, \fa) \to U_n (R) \to U_n (R/ \fa) \to 1.
$$}

\vskip1mm

\noindent {\rm (2)} \parbox[t]{16cm}{If, moreover, $K_2 (n, R/ \fa) \subset U_n (R/ \fa),$ then $K_2 (n,R) \subset U_n (R)$ and the natural sequence
$$
1 \to K_2 (n, R, \fa) \to K_2(n, R) \to K_2 (n, R/ \fa) \to 1
$$
is exact.}
\end{thm}

\vskip2mm

Using the theorem, one then proves
\begin{prop}\cite[Proposition 2.3]{IR1}\label{KT:P-BR2}
Suppose that $R$ is either a finite-dimensional algebra over an algebraically closed field $K$ or a finite ring with $2 \in R^{\times}.$ Then $K_2 (n, R) \subset U_n(R)$, and consequently $K_2 (n,R)$ is a central subgroup of $\mathrm{St}_n (R).$
\end{prop}

The proposition yields the following statement, which gives analogues in the noncommutative setting of Corollary \ref{C:StG-1} and Proposition \ref{P:StG-3}.

\begin{prop}\cite[Proposition 3.6]{IR1}\label{NC:P-St2}
Let $K$ be an algebraically closed field of characteristic 0 and $n$ and integer $\geq 3.$ Suppose $A$ is an associative algebraic ring over $K$ such that $2 \in A^{\times}$ and denote by $A^{\circ}$ the connected component of $0_A.$ Then

\vskip1mm

\noindent $\mathrm{(i)}$ $\mathrm{St}_n (A) = \mathrm{St}_n (A^{\circ}) \times P$, where $P$ is a finite group;

\vskip1mm

\noindent $\mathrm{(ii)}$ $K_2 (n, A^{\circ})$ is a central subgroup of $\mathrm{St}_n (A^{\circ}).$
\end{prop}

\vskip2mm

Finally, an important ingredient in the proof of Theorem 2 will be the following proposition, which provides a description of $K_2$ in the relevant case.


\begin{prop}\cite[Proposition 2.4]{IR1}\label{KT:P-DA}
Let $k$ and $K$ be fields of characteristic 0, with $K$ algebraically closed. Suppose that $D$ is a finite-dimensional central division algebra over $k,$ $A$ a finite-dimensional algebra over $K$, and $f \colon D \to A$ a ring homomorphism with Zariski-dense image. Then for $n \geq 3,$ $K_2 (n,A)$ coincides with the subgroup
$$
U'_n (A) = < c(u,v) \mid u,v \in \overline{f(L^{\times})} >,
$$
where $L$ is an arbitrary maximal subfield of $D$.
\end{prop}

One of the key steps in the proof of the proposition is the following lemma, which is also of independent interest.
\begin{lemma}\label{KT:L-DA}
Let $A,$ $D$, and $f$ be as above, and set
$C = \overline{f(k)}$ (Zariski closure). Then
\begin{equation}\label{E-LDA}
A \simeq D \otimes_k C \simeq M_{s} (C)
\end{equation}
as $K$-algebras, where $s^2 = \dim_k D.$ Moreover, if $L$ is any maximal subfield of $D$, then the second isomorphism can be chosen so that $L \otimes_k C \simeq D_s (C)$, where $D_s (C) \subset M_s (C)$ is the subring of diagonal matrices.
\end{lemma}

\vskip5mm

\section{Rigidity results over commutative rings}\label{S:T1, T2}

This section is devoted to the proof of Theorem 1.
For the reader's convenience, we would like to begin by sketching the general strategy of the argument. Suppose $\Phi$ is a reduced irreducible root system of rank $\geq 2$, $G$ the universal Chevalley-Demazure group scheme of type $\Phi$, and $R$ a commutative ring such that $(\Phi, R)$ is a nice pair\footnotemark \footnotetext{Recall that this means that $2 \in R^{\times}$ whenever $\Phi$ contains a subsystem of type $B_2$, and $2,3 \in R^{\times}$ if $\Phi$ is of type $G_2$.}.
Given an abstract representation $$\rho \colon G(R)^+ \to GL_n (K)$$ of the elementary subgroup $G(R)^+$ over an algebraically closed field $K$, we first show that one can associate to $\rho$ an algebraic ring $A$ over $K$, together with a homomorphism of abstract rings $f \colon R \to A$ having Zariski-dense image (this generalizes a construction of Kassabov and Sapir \cite{Kas} originally carried out for $SL_n$). The construction of $A$ relies on explicit computations with Steinberg commutator relations. Next, we lift $\rho$, in an appropriate sense, to a representation
$\tilde{\tau} \colon \St (\Phi, A) \to H$ of the corresponding Steinberg group, where $H = \overline{\rho (G(R)^+)}$ is the Zariski closure of the image of $\rho$.
Then, using information on the structure of $K_2 (\Phi, A^{\circ})$, where $A^{\circ}$ is the connected component of $A$,
we show that $\tilde{\tau} \vert_{\St (\Phi, A^{\circ})}$ descends to an abstract homomorphism
$\sigma \colon G(A^{\circ}) \to H$, such that the composition
$$
G(R)^+ \to G(A) \to G(A^{\circ}) \stackrel{\sigma}{\rightarrow} H,
$$
where the first map is induced by $f$ and the second by the natural projection $A \to A^{\circ}$ (cf. Propositions \ref{P:AR-4} and \ref{P:AR-2}), coincides with $\rho$ on a finite index subgroup $\Delta \subset G(R)^+.$ Finally,
in the situations considered in the statement of Theorem 1, $B = A^{\circ}$ is a finite-dimensional commutative $K$-algebra and, furthermore, we show that $\sigma$ is actually a morphism of algebraic groups.

\subsection{Algebraic ring associated to a representation}
The main result of this section is the following theorem, which enables us to attach an algebraic ring $A$ to an abstract representation $\rho \colon G(R)^+ \to GL_n (K).$

\begin{thm}\cite[Theorem 3.1]{IR}\label{T:ARR-1}
Let $\Phi$ be a reduced irreducible root system of rank $\geq 2$,  $G$ the corresponding universal Chevalley-Demazure group scheme, and $R$ a commutative ring such that $(\Phi, R)$ is a nice pair. Then, given a representation $\rho \colon G(R)^+ \to GL_n (K)$, there exists a commutative algebraic ring $A$ with identity, together with a homomorphism of abstract rings $f \colon R \to A$ having Zariski-dense image, such that for every root $\alpha \in \Phi$, there is an injective regular map $\psi_{\alpha} \colon A \to H$ into $H := \overline{\rho (G(R)^+)}$ (Zariski closure) satisfying
\begin{equation}\label{E:ARR101}
\rho (e_{\alpha} (t)) = \psi_{\alpha} (f(t)).
\end{equation}
for all $t \in R.$
\end{thm}

\vskip3mm


To make this result somewhat more concrete, let us now consider several examples of ring homomorphisms that can arise in this context. Note that in the results of Borel and Tits \cite{BT}, one only needs to consider embeddings of one field into another. Over rings, however, there are considerably more possibilities for homomorphisms.





\vskip2mm

\noindent {\bf Example 4.1.2.} Let $R = \Z[X]$ and $K$ be an algebraically closed field of characteristic 0.
Fix $s$ \emph{distinct} points $a_1, \dots, a_s \in K^{\times}.$

\vskip1mm

\noindent (a) The most basic example of a ring homomorphism is simply an evaluation homomorphism. Namely,
let $B = K \times \cdots \times K$ ($s$ copies), and define
$$
f \colon R \to B, \ \ \ \ g(X) \mapsto (g(a_1), \dots, g(a_s)).
$$
To see that $f$ has Zariski-dense image we argue as follows. First notice that since the image is obviously invariant under $f(\Z)$, which is infinite and contained in the diagonal $\{(a, \dots, a) \mid a \in K \}$, the closure of the image is a linear subspace of $B$. Furthermore, the elements
$$
f(1) = (1, \dots, 1), \  f(X) = (a_1, \dots, a_s),\ \dots, \ f(X^{s-1}) = (a_1^{s-1}, \dots, a_s^{s-1})
$$
are linearly independent as the Vandermonde determinant is nonzero, which yields Zariski-density.

\vskip2mm

\noindent (b) Now let
$B = K[\varepsilon_1] \times \cdots \times K[\varepsilon_s]$, with $\varepsilon_i^2 = 0$ for all $i.$ Define
$$
f \colon R \to B, \ \ \ g(X) \mapsto (g(a_1) + g'(a_1)\varepsilon_1, \dots, g(a_s) + g'(a_s) \varepsilon_s),
$$
where $g'(a_i)$ denotes the derivative of the polynomial $g(X)$ evaluated at $a_i.$ Set
$$
a = f(X) = (a_1 + \varepsilon_1, \dots, a_s + \varepsilon_s).
$$
Let $C = \overline{f(R)}$ be the Zariski closure of the image and consider the projection map
$$
\varphi \colon B \to K \times \cdots \times K, \ \ (\alpha_1 + \beta_1 \varepsilon_1, \cdots, \alpha_s + \beta_s \varepsilon_s) \mapsto (\alpha_1, \cdots, \alpha_s).
$$
It follows from (a) that $\varphi (C) = K \times \cdots \times K.$ We claim that $C$ contains the elements
$$
(0, \dots, 0, \varepsilon_i, 0, \dots, 0)
$$
for $i = 1, \dots, s.$ Since $C$ is a subspace, this will imply that $\ker \varphi \subset C$, and hence $C = B.$ We will give the argument only for $i = 1.$ First, we note that $C$ contains an element of the form
$$
(\varepsilon_1, *, \dots, *).
$$
Indeed, since $a_1 \neq 0$, this follows from the fact that
$$
a^2 - a_1 a = (a_1 \varepsilon_1, *, \dots, *).
$$
Next, since $(1, 0, \dots, 0) \in \varphi(C)$, we see that $C$ contains an element of the form
$$
b = (1 + \beta_1 \varepsilon_1, \beta_2 \varepsilon_2, \dots, \beta_s \varepsilon_s).
$$
Now,
$$
b^2 = (1 + 2 \beta_1 \varepsilon_1, 0, \dots, 0),
$$
so that
$$
(\varepsilon_1, *, \dots, *) \cdot (1 + 2 \beta_1 \varepsilon_1, 0, \dots, 0) = (\varepsilon_1, 0, \dots, 0) \in C,
$$
as needed.



\vskip2mm
\noindent (c) Let $B = K[\omega_n]$, with $\omega_n^{n+1} = 0,$ and define
$$
f \colon R \to B, \ \ \ g(X) \mapsto g(a_1) + g'(a_1) \omega_n + \frac{g''(a_1)}{2!} \omega_n^2 + \cdots + \frac{g^{(n)}(a_1)}{n!} \omega_n^n,
$$
where $g^{(k)} (a_1)$ denotes the $k^{\rm{th}}$ derivative of $g(X)$ evaluated at $a_1.$ It is easy to see that $f$ is a ring homomorphism. Now, since $f(1) = 1,$ we have $K \subset C = \overline{f(R)}$; then the fact that $f(X) = a_1 + \omega_n$ implies that $\omega_n \in C$, hence $C = B$.

\vskip2mm

\noindent {\bf Remark 4.1.3.} In \cite{LR}, a variation of example (c) was used to construct abstract homomorphisms of the groups of points of semisimple algebraic groups over fields of characteristic 0 such that the unipotent radical of the Zariski closure of the image has prescribed nilpotence class (this contains the example of Borel and Tits discussed in the introduction as a special case). The set-up is as follows. Let $K/k$ be an extension of fields of characteristic 0, with $K$ algebraically closed. Fix an integer $n \geq 1$ and let $B = K[\omega_n]$, with $\omega_n^{n+1} = 0$, as above. Let $\delta \colon k \to K$ be a nontrivial derivation and denote by $k_0 \subset k$ the subfield of constants of $\delta$ (i.e. $\delta(x) = 0$ for all $x \in k_0$). It is known that $\delta$ can be extended (not necessarily uniquely) to a derivation $K \to K$ (see \cite[Proposition 8.17]{J}); we fix one such extension and denote it also by $\delta$. Now define the {\it $n^{th}$ Taylor homomorphism} associated with (this extension of) $\delta$ by
$$
t_{\delta, n} \colon k \to K[\omega_n], \ \ \ x \mapsto x + \delta(x) \omega_n + \frac{\delta^2 (x)}{2!} \omega_n^2 + \cdots + \frac{\delta^n (x)}{n!} \omega_n^n
$$
(one verifies by direct computation that this is a homomorphism of $k_0$-algebras).
Using the fact that algebraic subgroups of $B$ coincide with $K$-subspaces, one can show that $t_{\delta, n}$ has Zariski-dense image (see \cite[Proposition 3(iii)]{LR}). Now, let $G$ be a connected semisimple algebraic group defined over $k_0$ and let $\G_n = R_{B/K}(_{B}G)$ be the $K$-group obtained by restriction of scalars; by construction, we have $G(B) = \G_n (K).$ Since $t_{\delta,n}$ is a homomorphism of $k_0$ algebras, it induces a group homomorphism
$$
\beta \colon G(k) \to G(B) = \G_n (K)
$$
on the groups of rational points. Using the unirationality of $G$ \cite[Theorem 18.2(ii)]{Bo}, one shows that $\beta$ has Zariski-dense image. Moreover, the unipotent radical of $\G_n$ has nilpotence class $n$ \cite[Proposition 4]{LR}.

\vskip3mm

\addtocounter{thm}{2}

Now, using the fact that the Weyl group $W(\Phi)$ of $\Phi$ acts transitively on roots of the same length, the proof of Theorem \ref{T:ARR-1} easily reduces to the cases where $\Phi$ is of type $A_2$, $B_2$, or $G_2.$ Furthermore, it is enough to construct a regular map  $\psi_{\alpha} \colon A \to H$ satisfying (\ref{E:ARR101}) for a single root of each length in $\Phi$.

To illustrate the strategy of the proof, but at the same time to avoid excessive technicalities, we would like to discuss the case of $\Phi$ of type $A_2$ (this case was handled by Kassabov and Sapir -- see \cite[Theorem 3]{Kas}). We will use the standard realization of $\Phi$, described in \cite{Bour}, where the roots are of the form $\varepsilon_i - \varepsilon_j$,
with $i,j \in \{1, 2, 3 \}, i \neq j.$ We will write $e_{ij} (t)$ to denote $e_{\alpha} (t)$ for $\alpha = \varepsilon_i - \varepsilon_j$. Notice that $e_{ij} (t)$ is simply the usual elementary matrix with $1$'s on the diagonal, $t$ in the $(i,j)$-th entry, and 0 elsewhere, and that $G(R)^+$ is the elementary subgroup $E_3 (R) \subset GL_3 (R)$. In particular, the set-up works for a general associative ring $R$.
So, let $\rho \colon E_3(R) \to GL_n(K)$ be a representation,
set $\alpha = \varepsilon_1 - \varepsilon_3,$ and define $A$ to be $A_{\alpha} = \overline{\rho (e_{\alpha} (R))}.$
We let $\ba \colon A \times A \to A$ be the restriction of the matrix product in $H = \overline{\rho (E_3(R))}$ to $A$ and note that
$(A, \ba)$ is a commutative algebraic subgroup of $H$.
Furthermore, we define $f = f_{\alpha} \colon R \to A$ to be the map
$t \mapsto \rho (e_{\alpha} (t)).$ Clearly
\begin{equation}\label{E:ARR202}
\ba (f(t_1), f(t_2)) = f(t_1 + t_2)
\end{equation}
for all $t_1, t_2 \in R.$ To define $\bm$, we need the following elements
$$
w_{12} = e_{12} (1) e_{21} (-1) e_{12} (1) = \left( \begin{array}{rcc} 0 & 1 & 0 \\ -1 & 0 & 0 \\ 0 & 0 & 1 \end{array} \right) \ \ \ \ {\rm and} \ \ \ \ w_{23} = e_{23} (1) e_{32} (-1) e_{23} (1) = \left( \begin{array}{crc} 1 & 0 & 0 \\ 0 & 0 & 1 \\ 0 & -1 & 0 \end{array} \right).
$$
It is easily checked that
\begin{equation}\label{E:ARR203}
w_{12}^{-1} e_{13}(r) w_{12} = e_{23} (r), \ \ \ w_{23} e_{13} (r) w_{23}^{-1} = e_{12} (r)
\end{equation}
and
\begin{equation}\label{E:ARR204}
[e_{12} (r), e_{23} (s)] = e_{13} (rs)
\end{equation}
for all $r, s \in R$, where $[g, h] = g h g^{-1} h^{-1}.$ Define the regular map $\bm \colon A \times A \to H$ by
$$
\bm (a_1, a_2) = [\rho(w_{23}) a_1 \rho(w_{23})^{-1}, \rho(w_{12})^{-1} a_2 \rho(w_{12})].
$$
It follows from relations (\ref{E:ARR203}) and (\ref{E:ARR204}) that
\begin{equation}\label{E:ARR205}
\bm (f(t_1), f(t_2)) = f(t_1 t_2).
\end{equation}
Then, in particular, $\bm (f(R) \times f(R)) \subset f(R),$ implying that $\bm (A \times A) \subset A$, and allowing us to view $\bm$ as a regular map $\bm \colon A \times A \to A.$ The fact that $(A, \ba, \bm)$ is a commutative algebraic ring is a consequence of the following elementary lemma.

\begin{lemma}\label{L:ARR-1}
Let $A$ be an affine  variety equipped with two regular maps $\ba \colon A \times A \to A$ and $\bm \colon A \times A \to A$. Assume that $(A, \ba)$ is a commutative algebraic group and that there exists a homomorphism $f \colon R \to A$ of an abstract unital ring $R$ into $A$ such that $\overline{f(R)} = A$ and
\begin{equation}\label{E:ARR201}
f(t_1 + t_2) = \ba (f(t_1), f(t_2)) \ \ \ {\rm and} \ \ \ f(t_1 t_2) = \bm (f(t_1), f(t_2))
\end{equation}
for all $t_1, t_2 \in R.$ Then $(A, \ba, \bm)$ is an algebraic ring with identity, which is also commutative if $R$ is a commutative ring.

\end{lemma}
\begin{proof}
Condition (I) of Definition 2.1.1 holds by our assumption. To verify (II), we observe that (\ref{E:ARR201}), in conjunction with the fact that multiplication in $R$ is associative, implies that the regular maps
$$
\beta_1 \colon A \times A \times A \to A, \ \ \ \beta_1 (x,y,z) = \bm (\bm (x,y), z),
$$
and
$$
\beta_2 \colon A \times A \times A \to A, \ \ \ \beta_2 (x,y,z) = \bm (x, \bm (y,z))
$$
coincide on the Zariski-dense subset $f(R) \times f(R) \times f(R) \subset A \times A \times A$. It follows that they coincide everywhere, yielding (II). All other conditions (including the fact that $1_A := f(1_R)$ is the identity element in $A$) are verified similarly.
\end{proof}

Finally, we note that by our construction, (\ref{E:ARR101}) obviously holds for the inclusion map $\psi_{\alpha} \colon A \to H$, as required.

When $\Phi$ is of type $B_2$ or $G_2$, the construction proceeds along similar lines, but is quite a bit more involved. In particular, one needs to consider commutative algebraic groups $A_{\alpha}$ and $A_{\beta}$, defined as above as Zariski closures of images of roots subgroups, for a short root $\alpha$ and long root $\beta.$
Using the Steinberg commutator relations, one shows that the underlying algebraic varieties are isomorphic, which then allows one to define multiplication operations that make $A_{\alpha}$ and $A_{\beta}$ into isomorphic algebraic rings. The upshot is that this procedure enables us to associate an algebraic ring to a representation $\rho$ of $G(R)^+$ that captures information about the images of \emph{all} root subgroups simultaneously.

\vskip5mm

\subsection{Lifting to Steinberg groups}
Let $\rho \colon G(R)^+ \to GL_n (K)$ be a representation with associated algebraic ring $A$. The next step of the proof involves constructing a representation $$\tilde{\tau} \colon \St (\Phi, A) \to GL_n (K)$$ of the Steinberg group such that $\tilde{x}_{\alpha} (a) \mapsto \psi_{\alpha} (a)$ for all $a \in A$, where $\psi_{\alpha}$ is the regular map from Theorem \ref{T:ARR-1}. More precisely, we have the following statement.


\begin{prop}\label{P:StG-1}
Let $(\Phi, R)$ be a nice pair, $K$ an algebraically closed field, and $$\rho \colon G(R)^+ \to GL_n (K)$$ a representation. Furthermore, let $A$ and $f \colon R \to A$ be the algebraic ring and ring homomorphism associated to $\rho$ that were constructed in Theorem \ref{T:ARR-1}. Then there exists a group homomorphism $$\tilde{\tau} \colon \St (\Phi, A) \to H \subset GL_n (K)$$ such that $\tilde{\tau} \colon \tilde{x}_{\alpha} (a) \mapsto \psi_{\alpha} (a)$ for all $a \in A$ and all $\alpha \in \Phi$. Consequently, $\tilde{\tau} \circ \tilde{F} = \rho \circ \pi_R,$ where $\tilde{F} \colon \St (\Phi, R) \to \St (\Phi, A)$ is the homomorphism induced by $f.$
\end{prop}
\begin{proof}
It follows from Definition 3.1.1 that to establish the existence of $\tilde{\tau},$ we only need to show that
relations (R1) and (R2) are satisfied if each of the generators $\tilde{x}_{\alpha} (a)$ is replaced by $\psi_{\alpha} (a)$. First, let $a = f(s), b= f(t)$, with $s, t \in R.$ Since $\rho$ is a homomorphism and by Theorem \ref{T:ARR-1}, $\rho \circ e_{\alpha} = \psi_{\alpha} \circ f$, we have, in view of the relation $e_{\alpha} (s) e_{\alpha} (t) = e_{\alpha} (s+t)$, that
$$
\psi_{\alpha} (a) \psi_{\alpha} (b) = \rho (e_{\alpha} (t) e_{\alpha} (s)) = \rho (e_{\alpha} (t+s)) = \psi_{\alpha} (a + b).
$$
Thus, the two regular maps $A \times A \to H$ given by
$$
(a, b) \mapsto \psi_{\alpha} (a) \psi_{\alpha} (b) \ \ \ {\rm and} \ \ \ (a,b) \mapsto \psi_{\alpha} (a + b)
$$
coincide on $f(R) \times f(R).$ The Zariski density of the latter in $A \times A$ implies that they coincide everywhere, yielding (R1).

Similarly, using the commutator relation (\ref{E:StG104}), we see that the two regular maps $A \times A \to H$ defined by
$$
(a, b) \mapsto [ \psi_{\alpha} (a), \psi_{\beta} (b)] \ \ \ {\rm and} \ \ \ (a,b) \mapsto \prod \psi_{i \alpha + j \beta} (N^{i,j}_{\alpha, \beta} a^i b^j)
$$
coincide on $f(R) \times f(R),$ hence on $A \times A$, and (R2) follows. Finally, the maps $\tilde{\tau} \circ \tilde{F}$ and $\rho \circ \pi_R$ both send $\tilde{x}_{\alpha} (s),$ $s \in R$, to $\psi_{\alpha} (f(s)) = \rho (e_{\alpha} (s)) = (\rho \circ \pi_R) (\tilde{x}_{\alpha} (s)),$ so they coincide on $\St (\Phi, R).$

\end{proof}

Thus, we obtain that the diagram formed by the solid arrows in
\begin{equation}\label{D:StG-1}
\xymatrix{ \St (\Phi, R) \ar[d]_{\pi_R} \ar[r]^{\tilde{F}} & \St (\Phi, A) \ar[rrdd]^{\tilde{\tau}} \ar[d]^{\pi_A} \\ G(R) \ar[rrrd]_{\rho} \ar[r]^{F} & G(A) \ar@{.>}[rrd]^{\tau} \\ & & & H \\}
\end{equation}
commutes. To complete the proof of Theorem 1, the crucial thing is to show that
$\tilde{\tau}$ basically descends to a homomorphism $\tau \colon G(A) \to H$ which makes the whole diagram commutative -- this will be made more precise below.

First, however, we would like to mention the following consequence of Proposition \ref{P:StG-1}.

\begin{cor}\label{C:StG-4}
Let $(\Phi, R)$ be a nice pair. Assume that there is an integer $m \geq 1$ such that $m R = \{ 0 \}.$ Let $K$ be an algebraically closed field such that $\mathrm{char} \: K$ does not divide $m$, and suppose moreover that $R$ is noetherian if $\mathrm{char} \: K > 0.$ If $\rho \colon G(R)^+ \to GL_n (K)$ is an abstract representation, then $\rho (G(R)^+)$ is finite.
\end{cor}
\begin{proof}
Let $A$ and $f \colon R \to A$ be the algebraic ring and ring homomorphism associated to $\rho$ constructed in Theorem \ref{T:ARR-1}. Then $mA = \{0 \}$ as $m f(R) = \{0 \}$ and $f(R)$ is Zariski-dense in $A$. In particular $m A^{\circ} = \{0 \}$, where $A^{\circ}$ is the connected component of $0_A.$
Now recall that by Proposition \ref{P:AR-20}, the quotient $A^{\circ} / J$ is a $K$-algebra, where $J$ denotes the Jacobson radical of $A^{\circ}.$ Then our assumption that $m$ is not divisible by $\mathrm{char} \: K$ forces $A^{\circ}/ J = \{ 0 \}$,  hence $A^{\circ} = \{ 0 \}.$ Therefore, by Corollary \ref{C:StG-2}, $\St (\Phi, A) = P,$ a finite group. On the other hand, by Proposition \ref{P:StG-1}, the representation $\rho$ factors through $\St (\Phi, A),$ which proves the finiteness of $\rho (G(R)^+).$
\end{proof}

\vskip5mm

\subsection{Passage to a subgroup of finite index} As we remarked above, the general strategy for completing the proof of Theorem 1 will be to construct a homomorphism $\tau \colon G(A) \to H$ which makes the diagram (\ref{D:StG-1}) commute. We now make this idea more precise.

Let $\Phi$ be a reduced irreducible root system of rank $\geq 2$, $K$ an algebraically closed field, and $R$ a commutative ring such that $(\Phi, R)$ is a nice pair. Assume that $R$ is noetherian if $\mathrm{char} \: K > 0.$
Furthermore, let $G$ be the universal Chevalley-Demazure group scheme of type $\Phi$, and consider a representation $\rho \colon G(R)^+ \to GL_n (K)$. By Theorem \ref{T:ARR-1}, we can associate to $\rho$ an algebraic ring $A$ and a ring homomorphism $f \colon R \to A$ with Zariski-dense image. Moreover,
our assumptions imply that we can write $A = A^{\circ} \oplus C$, with $C$ a finite ring (in fact, a finite quotient of $R$ --- see Propositions \ref{P:AR-4} and \ref{P:AR-2}).
By Proposition \ref{P:StG-1}, there exists a group homomorphism $\tilde{\tau} \colon \St (\Phi, A) \to GL_n (K)$ such that $\tilde{\tau} \circ \tilde{F} = \rho \circ \pi_R,$ where $\tilde{F} \colon \St(\Phi, R) \to \St (\Phi, A)$ is the homomorphism of Steinberg groups induced by $f$, and $\pi_R \colon \St (\Phi, R) \to G(R)$ is the canonical homomorphism. On the other hand, by Corollary \ref{C:StG-2}, $\St (\Phi, A) = \St (\Phi, A^{\circ}) \times P,$ where $P = \St (\Phi, C)$ is a finite group. So, $\tilde{\Delta} := \tilde{F}^{-1} (\St (\Phi, A^{\circ}))$ and $\Delta := \pi_R(\tilde{\Delta})$ are subgroups of finite index in $\St (\Phi, R)$ and $G(R)^+$, respectively, and moreover $F(\Delta) \subset G(A^{\circ}).$ Letting $\tilde{\sigma}$ denote the restriction of $\tilde{\tau}$ to $\St (\Phi, A^{\circ})$, we obtain that
the solid arrows in
\begin{equation}\label{D:FI-1}
\xymatrix{ \tilde{\Delta} \ar[d]_{\pi_R} \ar[r]^{\tilde{F}} & \St (\Phi, A^{\circ}) \ar[rrdd]^{\tilde{\sigma}} \ar[d]^{\pi_{A^{\circ}}} \\ \Delta \ar[rrrd]_{\rho} \ar[r]^{F} & G(A^{\circ}) \ar@{.>}[rrd]^{\sigma} \\ & & & H^{\circ} \\}
\end{equation}
form a commutative diagram. To complete the prove of Theorem 1, we show that there exists a group homomorphism $\sigma \colon G(A^{\circ}) \to H^{\circ}$ making the full diagram (\ref{D:FI-1}) commute.
Observe that if such $\sigma$ exists, then $\rho \colon G(R)^+ \to GL_n  (K)$ necessarily coincides on $\Delta$ with the composition $\sigma \circ T$, where $T \colon G(R)^+ \to G(A^{\circ})$ is the group homomorphism induced by the ring homomorphism $t \colon R \to A^{\circ}$, defined as the composition of $f \colon R \to A$ with the canonical projection $s \colon A \to A^{\circ}.$

In our arguments, we will need the following result of Matsumoto:
\begin{lemma}\label{L:FI-100}{\rm (cf. \cite{M}, Corollary 2)}
Let $\Phi$ be a reduced irreducible root system of rank $\geq 2$ and let $G$ be the corresponding universal Chevalley-Demazure group scheme. If $S$ is a semilocal commutative ring, then $G(S) = G(S)^+.$
\end{lemma}

In particular, we see that $G(A^{\circ}) = G(A^{\circ})^+$ since $A^{\circ}$ is semilocal by Lemma \ref{L:AR-1}, so the canonical homomorphism $\pi_{A^{\circ}} \colon \St (\Phi, A^{\circ}) \to G(A^{\circ})$ is surjective. Consequently, the existence of $\sigma$ is equivalent to the triviality of the restriction $\tilde{\sigma} \colon K_2 (\Phi, A^{\circ}) \to H^{\circ}.$

The existence of $\sigma$ in the case that $R$ is a semilocal ring is given by the following proposition.




\begin{prop}\label{P:FI-3}
Suppose $R$ is a commutative semilocal ring. Then there exists a group homomorphism $\sigma \colon G(A^{\circ}) \to H^{\circ}$ making the diagram {\rm (\ref{D:FI-1})} commute. Moreover, if $\mathrm{char} \: K = 0$, then $B:= A^{\circ}$ is a finite-dimensional $K$-algebra, and $\sigma$ can be viewed as a homomorphism $G(B) \to H^{\circ}$ such that the composition $\sigma \circ T$, where $T \colon G(R)^+ \to G(B)$ is induced by $t \colon R \to B$, coincides with $\rho$ on a subgroup $\Delta \subset G(R)^+$ of finite index.
\end{prop}
\begin{proof}
Observe that by Corollary \ref{C:AR-2}, $t(R^{\times})$ is Zariski-dense in $(A^{\circ})^{\times},$ where as above, $t \colon R \to A^{\circ}$ is the composition of the homomorphism $f \colon R \to A$ with the projection $s \colon A \to A^{\circ}.$
Let $\Lambda = R^{\times} \cap f^{-1} (A^{\circ} \times \{ 1 \}).$ Then $\Lambda$ has finite index in $R^{\times},$ so since $(A^{\circ})^{\times}$ is irreducible, we obtain that $f(\Lambda)$ is Zariski-dense in $(A^{\circ})^{\times} \times \{1 \}.$ Fix a long root $\alpha$, and for $u, v \in A^{\times}$, let $\{u, v \}_{\alpha}$ denote the corresponding Steinberg symbol. Clearly
$$
\tilde{\tau} ( \{u, v \}_{\alpha}) = H_{\alpha} (uv) H_{\alpha} (u)^{-1} H_{\alpha} (v)^{-1},
$$
where for $r \in A^{\times}$, we set
$$
H_{\alpha} (r) = {{W}}_{\alpha} (r) {{W}}_{\alpha} (-1) \ \ \ {\rm and} \ \ \ {{W}}_{\alpha} (r) = \psi_{\alpha} (r) \psi_{-\alpha} (-r^{-1}) \psi_{\alpha} (r).
$$
Now by Proposition \ref{P:AR-1}, the map $A^{\times} \to A^{\times}, t \mapsto t^{-1}$ is regular, hence the map
$\Theta \colon A^{\times} \times A^{\times} \to H,$ $(u,v ) \mapsto \tilde{\tau} (\{ u, v \}_{\alpha})$ is also regular. On the other hand, as we noted earlier,
for $a, b \in R^{\times},$ we have $h_{\alpha} (ab) = h_{\alpha} (a) h_{\alpha} (b)$ (see \cite{Stb1}, Lemma 28).
So, by Proposition \ref{P:StG-1},
$$
\tilde{\tau} (\{ f(a), f(b) \} ) = \rho (h_{\alpha} (ab) h_{\alpha} (a)^{-1} h_{\alpha} (b)^{-1}) = 1
$$
for all $a, b \in R^{\times}.$
Since the closure of $f(R^{\times})$ in $A^{\times}$ contains $(A^{\circ})^{\times} \times \{1 \},$ it follows that $\tilde{\tau}$ vanishes on $((A^{\circ})^{\times} \times \{ 1 \}) \times ((A^{\circ})^{\times} \times \{ 1 \} ).$ But according to Corollary \ref{C:StG-1}, $\ker \pi_{A^{\circ}}$ is generated by the Steinberg symbols $\{ u, v \}_{\alpha}$ for any fixed long root $\alpha \in \Phi.$ Thus, $\tilde{\sigma}$ vanishes on $\ker \pi_{A^{\circ}}$, implying that the required homomorphism $\sigma \colon G(A^{\circ}) \to H^{\circ}$ exists. The fact that $B := A^{\circ}$ is a finite-dimensional $K$-algebra if $\mathrm{char} \: K = 0$ follows from Proposition \ref{P:AR-2}.
\end{proof}

\vskip2mm

To obtain the required homomorphism $\sigma$ in the
case that either $H^{\circ}$ is reductive or $\mathrm{char} \: K = 0$ and the unipotent radical $U$ of $H^{\circ}$ is commutative, we first need to prove the somewhat weaker statement that there exists a homomorphism $\bar{\sigma} \colon G(A^{\circ}) \to \bar{H}$ such that $\bar{\sigma} \circ \pi_{A^{\circ}} = \nu \circ \tilde{\sigma},$ where $Z(H^{\circ})$ is the center of $H^{\circ},$ $\bar{H} = H^{\circ}/ Z(H^{\circ})$, and $\nu \colon H^{\circ} \to \bar{H}$ is the canonical map. Then we show that in these cases, $\bar{\sigma}$ actually lifts to the needed homomorphism $\sigma.$ This will require several additional facts about the structure of $H^{\circ}$, which we now summarize (see \cite[\S 5]{IR} for details).



\begin{prop}\label{P:FI-1}
The group $H^{\circ}$ coincides with $\tilde{\sigma}(\tilde{G}(A^{\circ}))$ and is its own commutator.
\end{prop}

Next, recall that if $G$ is a connected algebraic group over a field $k$, a \emph{Levi subgroup} $L \subset G$ is a connected $k$-subgroup such that $G = L \ltimes R_u G,$ where $R_u G$ is the unipotent radical of $G$ (such a semi-direct product decomposition is called a \emph{Levi decomposition} of $G$). Notice that since a Levi subgroup maps isomorphically onto $G/R_uG$, it is reductive. In fact, Levi subgroups are maximal among reductive subgroups and provide cross-sections to the quotient map $\pi \colon G \to G/R_u G.$ It is a well-known result of Mostow \cite{Mos} that if $\mathrm{char} \: k = 0,$ then Levi subgroups exist and are all conjugate. However, in positive characteristic, they need not exist, nor be conjugate (see \cite[Appendix A.6]{CGP} for examples).

The proposition now yields

\begin{cor}\label{C:FI-5}
If $H^{\circ}$ has a Levi decomposition $H^{\circ} = U \rtimes S$, where $U$ is the unipotent radical of $H^{\circ}$ and $S$ is reductive, then $S$ is automatically semisimple. In particular, if $H^{\circ}$ is reductive, then it is semisimple and hence the center $Z(H^{\circ})$ is finite.

\end{cor}
\begin{proof}
Since $H^{\circ} = [H^{\circ}, H^{\circ}]$, we have $S= [S, S],$ so $S$ is semisimple (\cite{Bo}, Corollary 14.2). In particular, if $H^{\circ}$ is reductive, then $H^{\circ} = S$ is semisimple, hence $Z(H^{\circ})$ is finite.
\end{proof}

To streamline the statements of some results later on, it will be convenient to introduce the following condition on $H^{\circ}.$ As in the corollary, let $U$ be the unipotent radical of $H^{\circ}$ and $Z(H^{\circ})$ be the center of $H^{\circ}$. We will say that $H^{\circ}$ satisfies condition (Z) if

\vskip3mm

\noindent (Z)\ \  \parbox[t]{15cm}{$Z(H^{\circ}) \cap U = \{ e \}.$}

\vskip3mm
\noindent The main result concerning (Z) is the following proposition.

\begin{prop}\label{P:FI-2} \ \
\newline {\rm (i)}\parbox[t]{16cm}{Suppose $H^{\circ}$ satisfies {\rm (Z)}. Then $Z(H^{\circ})$ is finite. Moreover, if $\mathrm{char} \: K = 0$, then $Z(H^{\circ})$ is contained in any Levi subgroup of $H^{\circ}$.}

\vskip1mm

\noindent {\rm (ii)}\parbox[t]{15.5cm}{Assume that $\mathrm{char} \: K = 0$ and that $U$ is commutative. Then $H^{\circ}$ satisfies {\rm (Z)}.}
\end{prop}

\vskip4mm

Now set $\bar{H} = H^{\circ} / Z(H^{\circ})$.
Since $Z(H^{\circ})$ is a closed normal subgroup of $H^{\circ}$, it follows from (\cite{Bo}, Theorem 6.8) that
$\bar{H}$ is an affine algebraic group and the canonical map $\nu \colon H^{\circ} \to \bar{H}$ is a morphism of algebraic groups. We let $\bar{\rho} = \nu \circ \rho.$ Since $H^{\circ} = \tilde{\sigma}(\St (\Phi, A^{\circ}))$ by Proposition \ref{P:FI-1}, and $K_2 (\Phi, A^{\circ}) = \ker \pi_{A^{\circ}}$ is a central subgroup of $\St (\Phi, A^{\circ})$ by Corollary \ref{C:StG-1}, it is clear that $\tilde{\sigma}$ descends to a homomorphism $\bar{\sigma} \colon G(A^{\circ}) \to \bar{H}$ such that $\bar{\sigma} \circ \pi_{A^{\circ}} = \nu \circ \tilde{\sigma}.$
Thus, we have the following.

\begin{prop}\label{P:FI-5}
If either $\mathrm{char} \: K = 0$ or $H^{\circ}$ is reductive, then $B := A^{\circ}$ is a finite-dimensional $K$-algebra, and $\bar{\sigma}$ can be viewed as a homomorphism $G(B) \to \bar{H}$ such that the composition $\bar{\sigma} \circ T$, where $T \colon G(R) \to G(B)$ is induced by $t \colon R \to B$, coincides with $\bar{\rho}$ on a subgroup $\Gamma \subset G(R)^+$ of finite index.
\end{prop}
All of the statements have been proved except for the fact that $B$ is a $K$-algebra if $H^{\circ}$ is reductive. This follows from Proposition \ref{P:AR-20} in view of the next lemma.
\begin{lemma}\cite[Lemma 5.7]{IR}\label{L:FI-1}
Let $\ell$ be the nilpotency class of $U$. If $J$ is the Jacobson radical of $A^{\circ},$ then $J^{\ell + 1} = \{0 \}.$ In particular, if $H^{\circ}$ is reductive, then $J = \{ 0 \}.$
\end{lemma}

\vskip5mm

\subsection{Rationality}\label{S:Rat} In this section, we will complete the proof of Theorem 1. We begin by showing that the abstract group homomorphisms $\sigma \colon G(B) \to H^{\circ}$ and $\bar{\sigma} \colon G(B) \to \bar{H}$ constructed in Propositions \ref{P:FI-3} and \ref{P:FI-5}, respectively, are in fact morphisms of algebraic groups. In the latter case, this then enables us to lift $\bar{\sigma}$ to a morphism of algebraic groups $\sigma \colon G(B) \to H^{\circ}$ that makes the diagram (\ref{D:FI-1}) commutative.

We should point out that in the statements concerning the rationality of $\sigma$ and $\bar{\sigma}$, we are implicitly using the fact that the functor of restriction of scalars $R_{B/K}$ enables us to endow $G(B)$ with a natural structure of a connected algebraic group over $K$. We refer the reader to \cite[Appendix A.5]{CGP}, \cite[Chapter I, \S 1, 6.6]{DG}, and \cite[Appendices 2 and 3]{Oes} for a general discussion of restriction of scalars. In the present context, all we need is the following: since $B$ is a commutative finite-dimensional $K$-algebra, $R_{B/K}(G)$ (or more precisely $R_{B/K}(_{B}G)$, where $_{B}G$ is obtained from $G$ by the base change $\Z \hookrightarrow B$) is an algebraic $K$-group such that $R_{B/K} (G)(K)$ can be naturally identified with $G(B)$, which yields a structure of an algebraic $K$-group on the latter. Also note that for each root $\alpha \in \Phi$, we have a morphism $R_{B/K} (e_{\alpha}) \colon R_{B/K} (\mathbb{G}_a) \to R_{B/K} (G)$. Now, $R_{B/K} (\mathbb{G}_a) (K) \simeq B$ is an irreducible $K$-variety. On the other hand, since $B$ is a finite-dimensional $K$-algebra, it follows from Lemma \ref{L:FI-100} that $G(B) = G(B)^+,$ i.e. the images $R_{B/K} (e_{\alpha}) (R_{B/K} (\mathbb{G}_a)(K))$ generate $R_{B/K} (G)(K)$.
The following well-known statement implies that $G(B)$ is a connected algebraic group over $K$.
\begin{prop}\label{P:Connect} \cite[Proposition 2.2]{Bo}
Let $f_i \colon V_i \to G$ $(i \in I)$ be a family of $k$-morphisms from irreducible $k$-varieties $V_i$ into a $k$-group $G$, and assume $e \in f_i V_i = W_i$ for each $i \in I$, where $e$ is the identity element of $G$. Put $M = \bigcup_{i \in I} W_i$ and denote by $\mathcal{A}(M)$ the smallest closed subgroup of $G$ containing $M$. Then $\mathcal{A}(M)$ is a connected subgroup of $G$ defined over $k$. Moreover, there is a finite sequence $(\alpha(1), \dots, \alpha(n))$ in $I$ such that $\mathcal{A}(M) = W_{\alpha(1)}^{e_1} \cdots W_{\alpha(n)}^{e_n}$, where each $e_i = \pm 1.$
\end{prop}


We now need to introduce some additional notations.
As before, let $\Phi$ be a reduced irreducible root system of rank $\geq 2$, and let $G$ be the universal Chevalley-Demazure group scheme of type $\Phi.$ For the remainder of this section, we fix an ordering on $\Phi,$ and
denote be $\Phi^+$ and $\Phi^-$ the corresponding subsystems of positive and negative roots, respectively; also, we let $\Pi \subset \Phi^+$ be a system of simple roots. Set $m = \vert \Phi^+ \vert = \vert \Phi^- \vert,$ and $\ell = \vert \Pi \vert.$ Next, let $U^+$ and $U^-$ be the standard subschemes associated to $\Phi^+$ and $\Phi^-$, respectively, and let $T$ be the usual maximal torus (cf. \cite{Chev}, \S 4).

Now let $R$ be a commutative ring such that $(\Phi, R)$ is a nice pair, and $K$ be
an algebraically closed field. As previously, we will assume
that $R$ is noetherian if $\mathrm{char} \: K > 0.$ Given a representation $\rho \colon G(R)^+ \to GL_n (K)$, let $A$ be the algebraic ring associated to $\rho$ (Theorem \ref{T:ARR-1}). From now on, we will assume that $B : = A^{\circ}$
is a finite-dimensional $K$-algebra, so that $G(A^{\circ})= G(B)$ has a natural structure of a connected algebraic $K$-group, as explained above. We note that $B$ is indeed a finite-dimensional $K$-algebra if either $\mathrm{char} \: K = 0$ or $H^{\circ}$ is reductive (see Proposition \ref{P:FI-5}), which are precisely the cases needed to complete the proof of Theorem 1.
As a matter of convention,
whenever we need to emphasize that we are working with the maps involving $A^{\circ}$ constructed in the previous subsection, we will use the notation $G(A^{\circ})$ rather than $G(B).$

\vskip2mm

The following well-known lemma (cf. \cite{Bo1}) describes the ``big cell'' of $G(B).$

\begin{lemma}\label{L:R-1}
The product map $p: U^- \times T \times U \to G$ gives an isomorphism onto an open subscheme $\Omega \subset G$ over $\Z$. Consequently, if $B$ is a finite-dimensional $K$-algebra, then $\Omega (B)$ is a Zariski-open subvariety of $G(B).$
\end{lemma}

The key step in proving that the homomorphisms $\sigma$ and $\bar{\sigma}$ are regular will be to show
that their restrictions to $\Omega(A^{\circ})$ are regular.

\begin{lemma}\label{L:R-2}
Let $\sigma \colon G(A^{\circ}) \to H^{\circ}$ and $\bar{\sigma} \colon G(A^{\circ}) \to \bar{H}$ be the group homomorphisms constructed in Propositions {\ref{P:FI-3}} and {\ref{P:FI-5}}, respectively. Then their restrictions $\sigma \vert_{\Omega (A^{\circ})}$ and $\bar{\sigma} \vert_{\Omega (A^{\circ})}$ to $\Omega(A^{\circ})$ are regular.
\end{lemma}
\begin{proof}
Recall that there exist isomorphisms of group schemes over $\Z$
$$
\omega^+ \colon (\mathbb{G}_a)^m \to U, \ \ \ \omega^- \colon (\mathbb{G}_a)^m \to U^-, \ \ \ {\rm and} \ \ \ \omega \colon (\mathbb{G}_m)^{\ell} \to T
$$
such that for any commutative ring $R$, we have
$$
\omega^+ ((u_{\alpha})_{\alpha \in \Phi^+}) = \prod_{\alpha \in \Phi^+} e_{\alpha} (u_{\alpha}), \ \ \ \omega^- ((u_{\alpha})_{\alpha \in \Phi^-}) = \prod_{\alpha \in \Phi^-} e_{\alpha} (u_{\alpha}), \ \ \ {\rm and} \ \ \ \omega ((t_{\alpha})_{\alpha \in \Pi}) = \prod_{\alpha \in \Pi} h_{\alpha} (t_{\alpha}),
$$
where the roots in $\Phi^+$ and $\Phi^-$ are listed in an arbitrary (but {\it fixed}) order (cf. \cite{Chev}, \S 4). Let
$$
\theta \colon (\bG_a)^m \times (\bG_m)^{\ell} \times (\bG_a)^m \to G
$$
be the composition of the isomorphism
$$
\omega^- \times \omega \times \omega^+ \colon (\bG_a)^m \times (\bG_m)^{\ell} \times (\bG_a)^m \to U^- \times T \times U
$$
with the product map $p \colon U^- \times T \times U \to G.$ Then it follows from Lemma \ref{L:R-1} that $\theta$ is an isomorphism onto the Zariski-open subscheme $\Omega \subset G$, and consequently $\theta$ induces an isomorphism of $K$-varieties
$$
\theta_{A^{\circ}} \colon (A^{\circ})^m \times ((A^{\circ})^{\times})^{\ell} \times (A^{\circ})^m \to \Omega (A^{\circ}).
$$
Next, let $\psi_{\alpha} \colon A^{\circ} \to H$, $\alpha \in \Phi$, be the regular maps constructed in Theorem \ref{T:ARR-1}, and $$H_{\alpha} \colon (A^{\circ})^{\times} \to H, \ \ \ \alpha \in \Phi$$ be the regular maps introduced in the proof of Proposition \ref{P:FI-3}.
Define the regular map
$$
\kappa \colon (A^{\circ})^m \times ((A^{\circ})^{\times})^{\ell} \times (A^{\circ})^m \to H
$$
by
$$
\kappa ((u_{\alpha})_{\alpha \in \Phi^-}, (t_{\alpha})_{\alpha \in \Pi}, (u_{\alpha})_{\alpha \in \Phi^+}) = \left( \prod_{\alpha \in \Phi^-} \psi_{\alpha} (u_{\alpha}) \right) \cdot \left( \prod_{\alpha \in \Pi} H_{\alpha} (t_{\alpha}) \right) \cdot \left( \prod_{\alpha \in \Phi^+} \psi_{\alpha} (u_{\alpha}) \right).
$$
By our construction, $\sigma (e_{\alpha} (a)) = \psi_{\alpha} (a)$ for any $a \in A^{\circ}$ and all $\alpha \in \Phi,$ and consequently $(\sigma (h_{\alpha} (a)) = H_{\alpha} (a)$ for all $a \in (A^{\circ})^{\times}$ and $\alpha \in \Phi.$ Thus we have the following commutative diagram
$$
\xymatrix{ & (A^{\circ})^m \times ((A^{\circ})^{\times})^{\ell} \times (A^{\circ})^m \ar[ld]_{\theta} \ar[rd]^{\kappa} \\ G(A^{\circ}) \supset \Omega(A^{\circ})  \ar[rr]^{\sigma} & & H}
$$
Since $\kappa$ is regular and $\theta$ is an isomorphism of algebraic varieties over $K$, we conclude that $\sigma \vert_{\Omega (A^{\circ})}$ is regular. A similar argument shows that $\bar{\sigma} \vert_{\Omega (A^{\circ})}$ is regular.

\end{proof}

We can now prove
\begin{prop}\label{P:R-1}
The homomorphisms $\sigma \colon G(B) \to H^{\circ}$ and $\bar{\sigma} \colon G(B) \to \bar{H}$ constructed in Propositions \ref{P:FI-3} and \ref{P:FI-5}, respectively, are morphisms of algebraic groups.
\end{prop}
Indeed, as we remarked above, under the hypotheses of Propositions \ref{P:FI-3} and \ref{P:FI-5}, $G(B)$ is a connected algebraic group over $K$. Thus, the proposition follows immediately from
Lemma \ref{L:R-2} and the following (elementary) lemma.
\begin{lemma}\label{L:R-3}
Let $K$ be an algebraically closed field and let $\mathscr{G}$ and $\mathscr{G}'$ be affine algebraic groups over $K$, with $\mathscr{G}$ connected. Suppose $f \colon \mathscr{G} \to \mathscr{G}'$ is an abstract group homomorphism\footnotemark and assume there exists a Zariski-open set $V \subset \mathscr{G}$
such that $\varphi := f \vert_V$ is a regular map. Then $f$ is a morphism of algebraic groups. \footnotetext{Here we tacitly identify $\mathscr{G}$ and $\mathscr{G}'$ with the corresponding groups $\mathscr{G}(K)$ and $\mathscr{G}'(K)$ of $K$-points.}
\end{lemma}

\vskip2mm

\noindent {\bf Remark 4.4.6.} In the noncommutative setting, we will establish analogous rationality results using a geometric argument (see Lemma \ref{NC:L-R1}), which could also be applied in the present case when working over a field of characteristic 0. However, since the use of the big cell makes the proof more direct, we decided to give the above argument here.

\vskip3mm

\addtocounter{thm}{1}


The final step in the proof of Theorem 1 involves showing that
in the cases under consideration,
the morphism of algebraic groups $\bar{\sigma} \colon G(B) \to \bar{H}$ can be lifted to a morphism $\sigma \colon G(B) \to H^{\circ}$ making the diagram (\ref{D:FI-1}) commutative.

For this, one needs several structural results for $G(B)$ as an algebraic $K$-group, where $B$ is an arbitrary commutative finite-dimensional $K$-algebra, which are also useful in other contexts.
Let $J = J(B)$ be the Jacobson radical of $B$. By the Wedderburn-Malcev Theorem (see \cite{P}, Theorem 11.6), there exists a semisimple subalgebra $\bar{B} \subset B$ such that $B = \bar{B} \oplus J$ as $K$-vector spaces and $\bar{B} \simeq B/J \simeq K \times \cdots \times K$ ($r$ copies) as $K$-algebras.
Let $e_i = (0, \dots, 0, 1, 0, \dots, 0) \in \bar{B}$ be the $i$th standard basis vector. Then we have $B = \oplus_{i=1}^r B_i,$ where $B_i = e_i B.$ Clearly, $B_i = \bar{B}_i \oplus J_i$ with $\bar{B}_i = e_i \bar{B} \simeq K$ and $J_i = e_i J$, so in particular, $B_i$ is a local $K$-algebra with maximal ideal $J_i.$ For an ideal $\mathfrak{b} \subset B$, we let $G(B, \mathfrak{b})$ denote the congruence subgroup modulo $\mathfrak{b},$ i.e. the kernel of the natural morphism of algebraic $K$-groups $G(B) \to G(B/ \mathfrak{b})$; it is clear that $G(B, \mathfrak{b})$ is a closed normal subgroup of $G(B).$ The main results about $G(B)$ are summarized in the following proposition.
\begin{prop}\label{P:R-3}
{\rm (i)} $G(B) = G(B,J) \rtimes G(\bar{B})$ is a Levi decomposition of $G(B)$;

\vskip2mm

\noindent {\rm (ii)} $G(B) \simeq G(B_1) \times \cdots \times G(B_r)$, where each $B_i$ is a finite-dimensional local $K$-algebra;

\vskip2mm

\noindent {\rm (iii)} \parbox[t]{15.6cm}{Suppose $B = K \oplus J$ is a finite-dimensional commutative local $K$-algebra. Let $d \geq 1$ be such that $J^d = \{ 0 \}$, and for each $k = 1, \dots, d-1,$ let $s_k = \dim_K J^k/ J^{k+1}.$ Then, for $k = 1, \dots, d-1,$ the quotient $\mathcal{G}_{k} := G(B, J^{k})/ G(B, J^{k +1})$ is isomorphic as an algebraic $K$-group
to $\mathfrak{g} \times \cdots \times \mathfrak{g}$ ($s_{k}$ copies), where $\mathfrak{g}$ is the Lie algebra of $G$, considered as an algebraic group in terms of the underlying vector space. Furthermore, the conjugation action of $G(K)$ on $\mathcal{G}_{k}$ is the sum of $s_{k}$ copies of the adjoint representation.}

\end{prop}

The following proposition completes the proof of Theorem 1.

\begin{prop}\label{P:R-2} In each of the following situations
\vskip1mm

\noindent {\rm (i)} $H^{\circ}$ is reductive (hence semisimple);

\vskip1mm

\noindent {\rm (ii)} $\mathrm{char} \: K = 0$ and $H^{\circ}$ satisfies condition {\rm (Z)}

\vskip1mm
\noindent there exists a morphism of algebraic groups $\sigma \colon G(A^{\circ}) \to H^{\circ}$ such that $\tilde{\sigma} = \sigma \circ \pi_{A^{\circ}},$ i.e. the diagram {\rm (\ref{D:FI-1})} commutes.
\end{prop}
\begin{proof}(Sketch)
By Proposition \ref{P:FI-5}, in both cases, $B:= A^{\circ}$ is a finite-dimensional $K$-algebra, so $G(A^{\circ}) = G(B)$ has a natural structure of a connected algebraic $K$-group. Furthermore, in each case, the center $Z(H^{\circ})$ is finite (see Corollary \ref{C:FI-5} and Proposition \ref{P:FI-2}), so the canonical morphism $\nu \colon H^{\circ} \to \bar{H}$ is a central isogeny.

According to Proposition \ref{P:R-1}, $\bar{\sigma} \colon G(A^{\circ}) \to \bar{H}$ is a morphism of algebraic groups, which, by our construction, satisfies $\nu \circ \tilde{\sigma} = \bar{\sigma} \circ \pi_{A^{\circ}}$, in the notations introduced earlier. As we already noted in the proof of Proposition \ref{P:FI-5}, in case (i), we have $J(A^{\circ}) = \{ 0 \}$, so $A^{\circ} \simeq K \times \cdots \times K$ (Proposition \ref{P:AR-20}), and therefore $G(A^{\circ}) = G(K) \times \cdots \times G(K)$ is a semisimple simply connected algebraic group. Then, according to (\cite{BT1}, Proposition 2.24), there exists a morphism of algebraic groups $\sigma \colon G(A^{\circ}) \to H$ such that
$$
\nu \circ \sigma = \bar{\sigma}.
$$
To show that such a $\sigma$ exists in case (ii), one argues in similar way, making use of the Levi decomposition in Proposition \ref{P:R-3}(i), together with the fact that $Z(H^{\circ})$ is contained in any Levi subgroup of $H^{\circ}$ (see Proposition \ref{P:FI-5}).


Then, in both cases it follows
from Proposition \ref{P:FI-5} that $\nu \circ \sigma \circ \pi_{A^{\circ}} = \nu \circ \tilde{\sigma}.$ Hence $$\chi \colon \St (\Phi, A^{\circ}) \to H^{\circ}$$ defined by
$$
\chi (g) = \tilde{\sigma}(g)^{-1} \cdot (\sigma \circ \pi_{A^{\circ}}) (g)
$$
has values in $Z(H^{\circ}).$ This, in conjunction with the fact that $\tilde{\sigma}$ and $\sigma \circ \pi_{A^{\circ}}$ are group homomorphisms, implies that $\chi$ is also a group homomorphism. However, since $\St (\Phi, A^{\circ})$ coincides with its commutator subgroup (\cite{St1}, Corollary 4.4), $\chi$ must be trivial, and therefore
$$
\sigma \circ \pi_{A^{\circ}} = \tilde{\sigma},
$$
as required.
\end{proof}

To summarize, we have proved the following theorem, which, in view of Proposition \ref{P:FI-2}, yields Theorem 1.
\vskip2mm

\begin{thm}\label{T:R-1}
Let $\Phi$ be a reduced irreducible root system of rank $\geq 2$, $R$ a commutative ring such that $(\Phi, R)$ is a nice pair, and $K$ an algebraically closed field. Assume that $R$ is noetherian if $\mathrm{char} \: K > 0.$ Furthermore let $G$ be the universal Chevalley-Demazure group scheme of type $\Phi$ and let $\rho \colon G(R)^+ \to GL_n (K)$ be a finite-dimensional linear representation of the elementary subgroup $G(R)^+ \subset G(R)$ over $K.$ Set $H = \overline{\rho (G(R)^+)}$ (Zariski closure), and let $H^{\circ}$ be the connected component of the identity of $H$. Then in each of the following situations
\vskip1mm

\noindent {\rm (1)} $H^{\circ}$ is reductive;

\vskip1mm

\noindent {\rm (2)} $\mathrm{char} \: K = 0$ and $R$ is semilocal;

\vskip1mm

\noindent {\rm (3)} $\mathrm{char} \: K = 0$ and $H^{\circ}$ satisfies condition {\rm (Z)}

\vskip1mm

\noindent there exists a commutative finite-dimensional $K$-algebra $B$, a ring homomorphism $f \colon R \to B$ with Zariski-dense image and a morphism $\sigma \colon G(B) \to H$ of algebraic $K$-groups such that for a suitable subgroup $\Delta \subset G(R)^+$ of finite index we have
$$
\rho \vert_{\Delta} = (\sigma \circ F) \vert_{\Delta},
$$
where $F \colon G(R)^+ \to G(B)^+$ is the group homomorphism induced by $f$.

\end{thm}

\vskip5mm

\subsection{Applications} We now give two examples showing how Theorem \ref{T:R-1} can be used to recover a couple of previously known rigidity results.

\vskip2mm

\noindent {\bf Example 4.5.1.} If $H^{\circ}$ is reductive, then it follows from Lemma \ref{L:FI-1} and our construction that $B$ can be chosen to have trivial Jacobson radical, and therefore $B \simeq K \times \cdots \times K$ ($r$ copies). Then the homomorphism $f \colon R \to B$ in Theorem \ref{T:R-1} is of the form $f = (f_1, \dots, f_m)$, where each component is a homomorphism
$f_i \colon R \to K.$ In particular, if $R = \Z [x_1, \dots, x_k]$, then each $f_i$ is just a specialization map. So, in this case, we obtain from Theorem \ref{T:R-1} that
any representation $\rho \colon G(R)^+ \to GL_n (K)$ coincides on a subgroup of finite index $\Delta \subset G(R)^+$ with $\sigma \circ F$, where $F = (F_1, \dots, F_r)$ and each component $F_i \colon G(R)^+ \to G(K)$ is induced by a specialization homomorphism, and $\sigma \colon G(K) \times \cdots \times G(K) \to H$ is a morphism of algebraic $K$-groups. Thus, Theorem \ref{T:R-1} generalizes the result of Shenfeld \cite{Sh} which treats the case $G= SL_n,$ $R = \Z[x_1, \dots, x_k]$, and $K= \C,$ using the centrality of the congruence kernel of $G = SL_n (R)$ established in \cite{KN} and mimicking the argument of Bass-Milnor-Serre \cite{BMS}.

\vskip4mm

\noindent {\bf Example 5.5.2.} Now consider the case where the unipotent radical $U$ of $H^{\circ}$ is commutative and that $\mathrm{char} \: K = 0.$ Then it follows from Lemma \ref{L:FI-1} and our construction that one can choose $B$ so that its Jacobson radical $J = J(B)$ satisfies $J^2 = \{ 0 \}.$ Moreover, it follows from Proposition \ref{P:R-3}(ii) that we can write
$$
G(B) = G(B_1) \times \cdots \times G(B_r),
$$
where each $B_i$ is a finite dimensional local $K$-algebra of the form $B_i = K \oplus J_i$ with $J_i^2 = \{ 0 \}.$ Hence it is enough to analyze the case where $B = K \oplus J$ with $J^2 = \{ 0 \}.$
So now choose a $K$-basis $\{ v_1, \dots, v_d \}$ of $J$. Then a homomorphism $f \colon R \to B$ can be written in the form
$$
f(r) = (f_0 (r), f_1(r) v_1 + \cdots + f_d (r) v_d),
$$
where $f_0 \colon R \to K$ is a ring homomorphism and the $f_i$'s, for $i \geq 1$, satisfy $$
f_i (r_1 r_2) = f_0 (r_1) f_i (r_2) + f_i (r_1) f_0 (r_2).
$$
Thus, each $f_i,$ $i \geq 1,$ is a derivation (with respect to $f_0$), and we recover, in a slightly different form, the result of L.~Lifschitz and A.~Rapinchuk \cite{LR}, which was established when $R = k$ is a field of characteristic zero.

\vskip5mm

\section{Rigidity results over noncommutative rings}\label{S:NC}
In this section, we will indicate how the strategy used in the proof of Theorem 1 can be modified
to obtain analogous rigidity results for the elementary groups $E_n (R)$ $(n \geq 3)$ over general associative rings.

\subsection{Reformulation of Theorem 2}
We would like to begin by giving an alternative  statement of Theorem 2, which can be generalized (in a somewhat weaker form) to (essentially) arbitrary associative rings.
First, note that if $B$ is a finite-dimensional algebra over an algebraically closed field $K$, then the elementary group $E_n (B)$ has the structure of a connected algebraic $K$-group. Indeed, using the regular representation of $B$ over $K$, it is easy to see that $GL_n (B)$ is a Zariski-open subset of $M_n (B),$ and hence an algebraic group over $K$. Now let us
view $B$ as a connected algebraic ring over $K$, and for $i, j \in \{1, \dots, n \},$ $i \neq j$, consider the regular maps
$$
\varphi_{ij} \colon B \to GL_n (B), \ \ \ b \mapsto e_{ij} (b)
$$
Set $W_{ij} = \mathrm{im} \ \varphi_{ij}.$ Then each $W_{ij}$ contains the identity matrix $I_n \in GL_n (B)$, and by definition $E_n (B)$ is generated by the $W_{ij}.$ So, $E_n (B)$ is a connected algebraic group by Proposition \ref{P:Connect}.

\begin{thm}\cite[Theorem 3.1]{IR1}\label{NC:T-NonComm1}
Suppose $k$ and $K$ are fields of characteristic 0, with $K$ algebraically closed, $D$ is a finite-dimensional central division algebra over $k$, and $n$ is an integer $\geq 3.$ Let $\rho \colon E_n (D) \to GL_m (K)$ be a finite-dimensional linear representation and set $H = \overline{\rho (E_n (D))}$ (Zariski closure). Then there exists a finite-dimensional associative $K$-algebra $\B$, a ring homomorphism $f \colon D \to \B$ with Zariski-dense image, and a morphism $\sigma \colon E_n (\B) \to H$ of algebraic $K$-groups such that
$$
\rho = \sigma \circ F,
$$
where $F \colon E_n (D) \to E_n (\B)$ is the group homomorphism induced by $f.$
\end{thm}

We also have the following result for general associative rings.

\begin{thm}\cite[Theorem 3.2]{IR1}\label{NC:T-NonComm2}
Suppose $R$ is an associative ring with $2 \in R^{\times},$ $K$ is an algebraically closed field of characteristic 0, and $n$ is an integer $\geq 3.$ Let $\rho \colon E_n (R) \to GL_m (K)$ be a finite-dimensional linear representation, set $H = \overline{\rho (E_n (R))}$, and denote by $H^{\circ}$ the connected component of $H$. If the unipotent radical of $H^{\circ}$ is commutative, there exists a finite-dimensional associative $K$-algebra $\B$, a ring homomorphism $f \colon R \to \B$ with Zariski-dense image, and a morphism $\sigma \colon E_n (\B) \to H$ of algebraic $K$-groups such that for a suitable finite-index subgroup $\Delta \subset E_n (R)$, we have
$$
\rho \vert_{\Delta} = (\sigma \circ F) \vert_{\Delta},
$$
where $F \colon E_n (R) \to E_n (\B)$ is the group homomorphism induced by $f.$
\end{thm}

\vskip2mm

Assuming Theorem \ref{NC:T-NonComm1}, let us now prove Theorem 2. Let $f \colon D \to \B$ be the ring homomorphism in Theorem \ref{NC:T-NonComm1}, and set $C = \overline{f(k)}.$ First note that by Lemma \ref{KT:L-DA}, we have $K$-algebra isomorphisms
$$
\B \simeq D \otimes_k C \simeq M_s (C)
$$
where $s^2 = \dim_k D.$ Consequently, $E_n (\B) \simeq E_n (M_s (C)) = E_{ns} (C).$ Moreover, since $C$ is a finite-dimensional $K$-algebra, in particular a semilocal commutative ring, $E_{ns} (C) = SL_{ns} (C)$ (see Lemma \ref{L:FI-100}). So, using the fact that $G = {\bf SL}_{n,D}$ is $K$-isomorphic to $SL_{ns}$ (\cite{PR}, 2.3.1),  we see that $E_n (\B) \simeq G(C).$ Letting $f_C \colon k \to C$ be the restriction of $f$ to $k$, we now obtain Theorem 2.

The proofs of Theorems \ref{NC:T-NonComm1} and \ref{NC:T-NonComm2} proceed along similar lines as the proof of Theorem 1 discussed above. The main new technical input consists of noncommutative analogues of Stein's result, which were discussed in \S \ref{S:KTNC}. Additionally, we adopt a somewhat different approach to establishing rationality statements (cf. Lemma \ref{NC:L-R1}).

For the sake of completeness, we now indicate the main points of the argument. Let $\rho \colon E_n (R) \to GL_m (K)$ be a representation. Recall that in our discussion of the case of $\Phi$ of type $A_2$ in Theorem \ref{T:ARR-1}, we showed that one can attach an (associative) algebraic ring $A$ to $\rho.$ The precise statement is as follows.


\begin{prop}\label{NC:P-AR}
Suppose $R$ is an associative ring, $K$ an algebraically closed field, and $n \geq 3$. Given a representation $\rho \colon E_n (R) \to GL_m (K)$, there exists an associative algebraic ring $A$, together with a homomorphism of abstract rings $f \colon R \to A$ having Zariski-dense image such that for all $i, j \in \{ 1, \dots, n \},$ $i \neq j$, there is an injective regular map $\psi_{ij} \colon A \to H$ into $H := \overline{\rho (E_n (R))}$ satisfying
\begin{equation}\label{E:AR-1}
\rho (e_{ij} (t)) = \psi_{ij} (f(t))
\end{equation}
for all $t \in R.$
\end{prop}

Moreover, by Remark 2.3.5, in the situation of Theorem \ref{NC:T-NonComm1}, the algebraic ring $A$ is automatically connected.

Next, one uses the same argument as in Proposition \ref{P:StG-1} to show that $\rho$ lifts to a representation of the Steinberg group $\St_n (A).$

\begin{prop}\label{NC:P-St1}
Suppose $R$ is an associative ring, $K$ an algebraically closed field, and $n \geq 3$, and let
$\rho \colon E_n (R) \to GL_m (K)$ be a representation. Furthermore, let $A$ and $f \colon R \to A$ be the algebraic ring and ring homomorphism constructed in Proposition \ref{NC:P-AR}.
Then there exists a group homomorphism $\tilde{\tau} \colon \mathrm{St}_n (A) \to H \subset GL_m (K)$ such that $\tilde{\tau} \colon x_{ij} (a) \mapsto \psi_{ij} (a)$ for all $a \in A$ and all $i,j \in \{1, \dots, n \}, i \neq j.$ Consequently, $\tilde{\tau} \circ \tilde{F} = \rho \circ \pi_R,$ where $\tilde{F} \colon \mathrm{St}_n (R) \to \mathrm{St}_n (A)$ is the homomorphism induced by $f.$
\end{prop}

Now, given a representation $\rho \colon E_n (R) \to GL_m (K),$ we let $f \colon R \to A$ be the ring homomorphism associated to $\rho$ (Proposition \ref{NC:P-AR}), and
denote by
$\tilde{F} \colon \mathrm{St}_n (R) \to \mathrm{St}_n (A)$ and $F \colon E_n (R) \to E_n (A)$ the group homomorphisms induced $f.$
Then under the hypotheses of Theorems \ref{NC:T-NonComm1} and \ref{NC:T-NonComm2}, we have $\mathrm{St}_n (A) = \mathrm{St}_n (A^{\circ})$ (Remark 2.3.5) and $\mathrm{St}_n (A) = \mathrm{St}_n (A^{\circ}) \times P$ (Proposition \ref{NC:P-St2}), respectively,
so in both cases $\tilde{\Delta} := \tilde{F}^{-1} (\mathrm{St}_n (A^{\circ}))$ and $\Delta : = \pi_R (\tilde{\Delta})$ are finite-index subgroups of  $\mathrm{St}_n (R)$ and $E_n (R)$. Moreover, it is clear that
$F(\Delta) \subset E_n (A^{\circ})$.
Thus, letting
$\tilde{\sigma}$ denote the restriction of $\tilde{\tau}$ to $\mathrm{St}_n (A^{\circ})$, we see that the solid arrows in
\begin{equation}\label{NC:D-3}
\xymatrix{\tilde{\Delta} \ar[r]^{\tilde{F}} \ar[d]_{\pi_R} & \mathrm{St}_n (A^{\circ}) \ar[d]^{\pi_{A^{\circ}}} \ar[rrdd]^{\tilde{\sigma}} \\ \Delta \ar[rrrd]_{\rho} \ar[r]^{F} & E_n(A^{\circ}) \ar@{.>}[rrd]^{\sigma} \\ & & & H^{\circ} \\}
\end{equation}
form a commutative diagram. As in the commutative case, one completes the proof by showing that
there exists a group homomorphism $\sigma \colon E_n (A^{\circ}) \to H^{\circ}$ (in fact, a morphism of algebraic groups) making the full diagram (\ref{NC:D-3}) commute.

In the setting of Theorem \ref{NC:T-NonComm1}, the existence of the required {\it abstract} homomorphism $\sigma$ is given by the next proposition.

\begin{prop}\label{NC:P-SL}
Suppose $k$ and $K$ are fields of characteristic 0, with $K$ algebraically closed, $D$ is a finite-dimensional central division algebra over $k$, and $n$ is an integer $\geq 3.$ Let $$\rho \colon E_n (D) \to GL_m (K)$$ be a representation and denote by $A$ the algebraic ring associated to $\rho$.
Then $A = A^{\circ}$ is a finite-dimensional $K$-algebra and there exists a homomorphism of abstract groups $\sigma \colon E_n (A^{\circ}) \to H^{\circ}$ making the diagram $\mathrm{(\ref{NC:D-3})}$ commute.
\end{prop}
For the proof, one uses Proposition \ref{KT:P-DA} to imitate the argument given in the proof of Proposition \ref{P:FI-3}.

\vskip2mm
To handle the situation of Theorem \ref{NC:T-NonComm2}, one again first needs to show that
there exists a group homomorphism $\bar{\sigma} \colon E_n (A^{\circ}) \to \bar{H}$ such that $\bar{\sigma} \circ \pi_{A^{\circ}} = \nu \circ \tilde{\sigma}$, where $Z(H^{\circ})$ is the center of $H^{\circ}$, $\bar{H} = H^{\circ}/ Z(H^{\circ})$, and $\nu \colon H^{\circ} \to \bar{H}$ is the canonical map. As in the commutative case, this relies on the centrality of $K_2 (n, A^{\circ})$ in $\mathrm{St}_n (A^{\circ})$ (cf. Proposition \ref{KT:P-BR2}).
\begin{prop}\label{NC:P-FI}
Suppose $R$ is an associative ring with $2 \in R^{\times}$, $K$ is an algebraically closed field of characteristic 0,
and $n \geq 3.$ Let $\rho \colon E_n (R) \to GL_m (K)$ be a representation, set $H = \overline{\rho (E_n (R))}$, and denote by $A$ the algebraic ring associated to $\rho.$  Then $A^{\circ}$ is a finite-dimensional $K$-algebra and
there exists a homomorphism
$\bar{\sigma} \colon E_n (A^{\circ}) \to \bar{H}$ such that $\bar{\sigma} \circ \pi_{A^{\circ}} = \nu \circ \tilde{\sigma}.$
\end{prop}

Finally, to conclude the proof, we show that the homomorphisms $\sigma$ and $\bar{\sigma}$ are in fact morphisms of algebraic groups; in the latter case, this allows us to lift $\bar{\sigma}$ to a
morphism of algebraic groups $\sigma \colon E_n (A^{\circ}) \to H^{\circ}$ making the diagram (\ref{NC:D-3}) commute.
Since the lifting of $\bar{\sigma}$ is accomplished in essentially the same way as in the commutative case, we will only address the proof of rationality. The argument is
based on the following geometric lemma.
\begin{lemma}\label{NC:L-R1}
Let $X, Y, Z$ be irreducible varieties over an algebraically closed field $K$ of characteristic 0. Suppose $s \colon X \to Y$ and $t \colon X \to Z$ are regular maps, with $s$ dominant, such that for any $x_1, x_2 \in X$ with $s(x_1) = s(x_2),$ we have $t(x_1) = t(x_2).$ Then there exists a rational map $h \colon Y \dashrightarrow Z$ such that $h \circ s= t$ on a suitable open subset of $X$.
\end{lemma}
\begin{proof}
Let $W \subset X \times Y \times Z$ be the subset
$$
W = \{ (x,y, z) \colon s(x) = y, t(x) = z \}.
$$
Notice that $W$ is the graph of the morphism
$$
\varphi \colon X \to Y \times Z, \ \ \ x \mapsto (s(x), t(x)),
$$
so $W$ is an irreducible variety isomorphic to $X$. Now consider the projection $\mathrm{pr}_{Y \times Z} \colon X \times Y \times Z \to Y \times Z$, and let $U = \mathrm{pr}_{Y \times Z} (W)$ and $V = \bar{U}$, where the bar denotes the Zariski closure. Then $V$ is an irreducible variety. Moreover, $U$ is constructible by (\cite{H2}, Theorem 4.4), so in particular contains a dense open subset $P$ of $V$, which is itself an irreducible variety. Let now $p \colon P \to Y$ be the projection to the first component. We claim that $p$ gives a birational isomorphism between $P$ and $Y$. From our assumptions, we see that $p$ is dominant, and since $\mathrm{char} \: K = 0,$ $p$ is also separable. So, it follows from (\cite{H2}, Theorem 4.6) that to show that $p$ is birational, we only need to verify that it is injective. Suppose that $u_1 = (y_1, z_1), u_2 = (y_2, z_2) \in P$ with $y_1 = y_2.$ By our construction, there exist $x_1, x_2 \in X$ such that $s(x_1) = y_1$, $t(x_1) = z_1$ and $s(x_2) = y_2, t(x_2) = z_2.$ Since $s(x_1) = s(x_2),$ we have $t(x_1) = t(x_2),$ so $u_1 = u_2$, as needed.

Since $p$ is birational, we can now take $h = \pi_Z \circ p^{-1}\colon Y \dashrightarrow Z,$ where $\pi_Z \colon Y \times Z \to Z$ is the projection.
\end{proof}

\vskip2mm

Let $\rho \colon E_n (R) \to GL_m (K)$ be a representation as in Theorem \ref{NC:T-NonComm1} or \ref{NC:T-NonComm2} and denote by $A$ the algebraic ring associated to $\rho.$ Also let $M$ be the set of all pairs $(i,j)$ with $1 \leq i, j \leq n$, $i \neq j.$ Then, as we observed at the beginning of this section, $E_n (A^{\circ})$ is the connected algebraic group generated by the images $W_{q} = \mathrm{im} \ \varphi_{q}$ of the regular maps
$$
\varphi_{q} \colon A^{\circ} \to GL_n (A^{\circ}), \ \ \ a \mapsto e_{q} (a),
$$
for all $q \in M.$
In particular, Proposition \ref{P:Connect} implies that there exists a finite sequence $(\alpha(1), \dots, \alpha(v))$, with $\alpha(i) \in M$, such that
$$
E_n (A^{\circ}) = W^{\varepsilon_1}_{\alpha(1)} \cdots W^{\varepsilon_v}_{\alpha(v)},
$$
where each $\varepsilon_i = \pm 1.$ Let
$$
X = \prod_{i=1}^v (A^{\circ})_{\alpha(i)}
$$
be the product of $v$ copies of $A^{\circ}$ indexed by the $\alpha(i)$ and define a regular map $s \colon X \to E_n (A^{\circ})$ by
\begin{equation}\label{NC:E-Reg}
s(a_{\alpha(1)}, \dots, a_{\alpha(v)}) = \varphi_{\alpha(1)}(a_{\alpha(1)})^{\varepsilon_1} \cdots \varphi_{\alpha(v)} (a_{\alpha(v)})^{\varepsilon_v}.
\end{equation}
Also let
\begin{equation}\label{NC:E-Reg1}
 t\colon X \to H^{\circ} \ \ \ t(a_{\alpha(1)}, \dots, a_{\alpha(v)}) =  \psi_{\alpha(1)}(a_{\alpha(1)})^{\varepsilon_1} \cdots \psi_{\alpha(v)} (a_{\alpha(v)})^{\varepsilon_v},
\end{equation}
where the $\psi_{\alpha(i)}$ are the regular maps from Proposition \ref{NC:P-AR}. With this set-up, we can now prove

\begin{prop}\label{NC:P-R2}
The homomorphisms $\sigma \colon E_n(A^{\circ}) \to H^{\circ}$ and $\bar{\sigma} \colon E_n (A^{\circ}) \to \bar{H}$
constructed in Propositions \ref{NC:P-SL} and \ref{NC:P-FI}, respectively, are morphisms of algebraic groups.
\end{prop}
\begin{proof}
We will only consider $\sigma$ as the argument for $\bar{\sigma}$ is completely analogous. Set
$Y = E_n (A^{\circ})$ and $Z = H^{\circ}$, and let $s \colon X \to Y$ and $t \colon X \to Z$ be the regular maps defined in (\ref{NC:E-Reg}) and (\ref{NC:E-Reg1}).
From the construction of $\sigma,$ it is clear that $(\sigma \circ s) (x) = t(x),$ so in particular
$s(x_1) = s(x_2)$ for $x_1, x_2 \in X$ implies that $t(x_1) = t(x_2).$ Hence, by Lemma \ref{NC:L-R1}, $\sigma$ is a rational map. Therefore, there exists an open subset $V \subset E_n (A^{\circ})$ such that $\sigma \vert_V$ is regular. So, it follows from Lemma \ref{L:R-3} that $\sigma \colon E_n (A^{\circ}) \to H^{\circ}$ is a morphism.

\end{proof}

\section{Applications to character varieties and deformations of representations}\label{S:T-3}

In this section, we will discuss the proof of Theorem 3 (we refer the reader to \cite{IR1} for the full details). Recall that our goal is to estimate the dimension of the character variety $X_n (\Gamma)$, where $\Gamma = G(R)^+$, $G$ is the universal Chevalley-Demazure group scheme corresponding to a reduced irreducible root system of rank $\geq 2$, and $R$ is a finitely generated commutative ring such that $(\Phi, R)$ is a nice pair. Our strategy will be first to
exploit the well-known connection, going back to A.~Weil \cite{W}, between the tangent space of $X_n (\Gamma)$ at a given point and the space of 1-cohomology of $\Gamma$ with coefficients in a naturally associated representation. We can then apply the results on standard descriptions of representations of $\Gamma$ given by Theorem 1 to relate the latter space to a certain space of derivations of the ring $R$ (cf. Proposition \ref{D:P-1}), which yields the required bound on $\dim X_n (\Gamma).$ At the end of this section, we will also indicate some applications of our results to various classical forms of rigidity.
Throughout this section, we will work over a fixed algebraically closed field $K$ of characteristic 0.

\vskip3mm

\subsection{Background on representation and character varieties}
In this section, we summarize, for the reader's convenience, some key definitions and basic properties related to representation and character varieties (further details can be found in the monograph of Lubotzky and Magid \cite{LM}).

Let $\Gamma$ be a finitely generated group and fix an integer $n \geq 1.$ Recall that the $n^{\mathrm{th}}$ {\it representation scheme} of $\Gamma$ is the functor $\fR_n (\Gamma)$ from the category of commutative $K$-algebras to the category of sets defined by
$$
\fR_n (\Gamma) (A) = \Hom (\Gamma, GL_n (A)).
$$
More generally, if $\mathcal{G}$ is a linear algebraic group over $K$, we let the representation scheme of $\Gamma$ with values in $\G$ be the functor $\fR (\Gamma, \G)$ defined by
$$
\fR (\Gamma, \G)(A) = \Hom (\Gamma, \G(A)).
$$
Using the fact that for any commutative $K$-algebra $A$, a homomorphism $\rho \colon \Gamma \to GL_n (A)$ is determined by the images of the generators, subject to the defining relations of $\Gamma$, one shows
that $\fR_n (\Gamma)$ is an affine $K$-scheme represented by a finitely-generated $K$-algebra $\mathfrak{A}_n (\Gamma).$ Similarly, $\fR (\Gamma, \G)$ is an affine $K$-scheme represented by a finitely-generated $K$-algebra $\mathfrak{A} (\Gamma, \G)$ (cf. \cite{LM}, Proposition 1.2). The set $\fR_n (\Gamma)(K)$ of $K$-points of $\fR_n (\Gamma)$ is then denoted by $R_n (\Gamma)$, and is called the $n^{\mathrm{th}}$~{\it representation variety} of $\Gamma.$ It is an affine variety over $K$ with coordinate ring $A_n (\Gamma) = \mathfrak{A}_n (\Gamma)_{\mathrm{red}},$ the quotient of $\mathfrak{A}_n (\Gamma)$ by its nilradical. The representation variety $R(\Gamma, \G)$ is defined analogously.

Now let $\rho_0 \in R(\Gamma, \G).$ To describe the Zariski tangent space of
$\fR (\Gamma, \G)$ at $\rho_0$, denoted by $T_{\rho_0} (\fR (\Gamma, \G))$, we will use the algebra of dual numbers $K[\varepsilon]$ (where $\varepsilon^2 = 0$). More specifically, it is well-known that $\fR (\Gamma, \G) (K[\varepsilon])$ is the tangent bundle of $\fR (\Gamma, \G)$, and therefore $T_{\rho_0} (\fR (\Gamma, \G))$
can be identified with the fiber over $\rho_0$ of the map $\mu \colon \fR (\Gamma, \G) (K[\varepsilon]) \to \fR (\Gamma, \G) (K)$ induced by the augmentation homomorphism $K[\varepsilon] \to K, \varepsilon \mapsto 0$ (cf. \cite{Bo}, AG 16.2).
In other words, we have
$$
T_{\rho_0}(\fR (\Gamma, \G)) =  \{ \rho \in \mathrm{Hom} (\Gamma, \G(K[\varepsilon])) \mid \mu \circ \rho = \rho_0 \}.
$$
For our purposes, it will be useful to have the following alternative description of $T_{\rho_0} (\fR (\Gamma, \G)).$
Let $\tilde{\g}$ be the Lie algebra of $\G.$ Notice that $\tilde{\g}$ has a natural $\Gamma$-action given by
$$
\gamma \cdot x = \mathrm{Ad}(\rho_0 (\gamma)) x,
$$
for $\gamma \in \Gamma$ and $x \in \tilde{\g}$, where $\mathrm{Ad} \colon \G(K) \to GL(\tilde{\g})$ is the adjoint representation.
It is well-known that
$T_{\rho_0} (\fR (\Gamma, \G))$ can be identified with the space $Z^1 (\Gamma, \tilde{\g})$ of 1-cocycles (cf. \cite{LM}, Proposition 2.2). Indeed, an element $c \in Z^1 (\Gamma, \tilde{\g})$ is by definition a map $c \colon \Gamma \to \tilde{\g}$ such that
$$
c (\gamma_1 \gamma_2) = c(\gamma_1) + \mathrm{Ad}(\rho_0 (\gamma_1)) c(\gamma_2).
$$
On the other hand, we have an isomorphism
$\G (K[\varepsilon]) \simeq \tilde{\g} \rtimes \G$ given by
$$
B + C \varepsilon \mapsto (CB^{-1}, B).
$$
Hence an element $\rho \in  T_{\rho_0}(\fR(\Gamma, \G))$ is a homomorphism $\rho \colon \Gamma \to \tilde{\g} \rtimes \G$ whose projection to the second factor is $\rho_0.$ In other words, it arises from a map
$c \colon \Gamma \to \tilde{\g}$ such that the map
$$
\Gamma \to \tilde{\g} \rtimes \G, \ \ \ \  \gamma \mapsto (c (\gamma), \rho_0(\gamma))
$$
is a group homomorphism. With the above identification, the group homomorphism condition translates into the fact that
$$
c (\gamma_1 \gamma_2) = c(\gamma_1) + \mathrm{Ad}(\rho_0 (\gamma_1)) c(\gamma_2),
$$
giving the required isomorphism of $T_{\rho_0} (\fR (\Gamma, \G))$ with $Z^1 (\Gamma, \tilde{\g}).$
Also notice that for any finite-index subgroup $\Delta \subset \Gamma$ (which is automatically finitely generated), the natural restriction maps $\fR (\Gamma, \G) \to \fR (\Delta, \G)$ and $Z^1 (\Gamma, \tilde{\g}) \to Z^1 (\Delta, \tilde{\g})$ induce a commutative diagram
$$
\xymatrix{T_{\rho_0} (\fR (\Gamma, \G)) \ar[r] \ar[d] & Z^1 (\Gamma, \tilde{\g}) \ar[d] \\ T_{\rho_0} (\fR (\Delta, \G)) \ar[r] & Z^1 (\Delta, \tilde{\g})}
$$
where the horizontal maps are the isomorphisms described above.

Next, let us recall a characterization of the space $B^1 (\Gamma, \tilde{\g})$ of 1-coboundaries that will be used later; for this, we need to consider the action of $\G(K)$ on $R (\Gamma, \G).$ Given $\rho_0 \in R (\Gamma, \G)$,
let $\psi_{\rho_0} \colon \G(K) \to R (\Gamma, \G)$ be the orbit map, i.e. the map defined by
$$
\psi_{\rho_0}(T) = T \rho_0 T^{-1}, \ \ \  T \in \G (K).
$$
By direct computation, one shows that under the isomorphism
$T_{\rho_0} (\fR (\Gamma, \G)) \simeq Z^1 (\Gamma, \tilde{\g})$, the image of the differential
$(d \psi_{\rho_0})_e \colon T_{e} (\G) \to T_{\rho_0} (R (\Gamma, \G)) \subset T_{\rho_0} (\fR (\Gamma, \G))$
consists of maps $\tau \colon \Gamma \to \tilde{\g}$ such that there exists $A \in \tilde{\g}$ with
$$
\tau (\gamma) = A - \mathrm{Ad}(\rho_0 (\gamma)) A
$$
for all $\gamma \in \Gamma$, i.e. the image coincides with
$B^1 (\Gamma, \tilde{\g})$ (cf. \cite{LM}, Proposition 2.3). In fact, if $O (\rho_0)$ is the orbit of $\rho_0$ in $R(\Gamma, \G)$ under the action of $\G(K)$, then $B^1 (\Gamma, \tilde{\g})$ can be identified with $T_{\rho_0} (O (\rho_0)) \subset T_{\rho_0} (R(\Gamma, \G))$ (\cite{LM}, Corollary 2.4).

As a special case of the above constructions, we can consider
the action of $GL_n (K)$ on $R_n (\Gamma).$ The $n^{\mathrm{th}}$ {\it character variety} of $\Gamma$, denoted $X_n (\Gamma)$, is by definition the (categorical) quotient of $R_n (\Gamma)$ by $GL_n (K)$, i.e. it is the affine $K$-variety with coordinate ring $A_n (\Gamma)^{GL_n (K)}.$ Let $\pi \colon R_n (\Gamma) \to X_n (\Gamma)$ be the canonical map. Then each fiber $\pi^{-1} (x)$ contains a semisimple representation, and moreover if $\rho_1, \rho_2 \in R_n (\Gamma)$ are semisimple with $\pi (\rho_1) = \pi (\rho_2),$ then $\rho_1 = T \rho_2 T^{-1}$ for some $T \in GL_n (K).$ In particular, we see that $\pi$ induces a bijection between the isomorphism classes of semisimple representations and the points of $X_n (\Gamma)$ (cf. \cite{LM}, Theorem 1.28).

\vskip5mm

\subsection{A reduction step in the proof of Theorem 3} In this section, we would like to explain the reduction of the proof of Theorem 3 to the estimations of the dimensions of certain cohomology groups. The main technical statement will be given in Proposition \ref{D:P-1} below, and Theorem 3 easily follows from it.

As before, let $\Phi$ be a reduced irreducible root system of rank $\geq 2$, $G$ the universal Chevalley-Demazure group scheme of type $\Phi$, and $R$ a finitely generated commutative ring such that $(\Phi, R)$ is a nice pair. Throughout this section, we let $\Gamma = G(R)^+$ be the elementary subgroup of $G(R).$

Now, as we mentioned earlier, it has been shown by Ershov, Jaikin, and Kassabov \cite{EJK} that the group $\Gamma$ has Kazhdan's property (T). In particular, $\Gamma$ is finitely generated and satisfies the condition

\vskip2mm

\noindent (FAb) \ \ \parbox[t]{15cm}{{\it for any finite-index subgroup $\Delta \subset \Gamma,$ the abelianization $\Delta^{\mathrm{ab}} = \Delta / [\Delta, \Delta]$ is finite}}

\vskip2mm
\noindent (cf. \cite{HV}). This has the following consequence.
\begin{prop}\label{D:P-KG}
\cite[Proposition 2]{AR} Let $\Gamma$ be a group satisfying {\rm (FAb)}. For any $n \geq 1$, there exists a finite collection $G_1, \dots, G_d$ of algebraic subgroups of $GL_n (K)$, such that for any completely reducible representation $\rho \colon \Gamma \to GL_n (K)$, the Zariski closure $\overline{\rho (\Gamma)}$ is conjugate to one of the $G_i.$ Moreover, for each $i$, the connected component $G_i^{\circ}$ is a semisimple group.
\end{prop}

Consequently, for any $n \geq 1$, we have
\begin{equation}\label{E:D-SS1}
X_n (\Gamma) = \bigcup_{i=1}^d \pi (R' (\Gamma, G_i)),
\end{equation}
where $\pi \colon R_n (\Gamma) \to X_n (\Gamma)$ is the canonical map, and for all $i$, $R' (\Gamma, G_i) \subset R (\Gamma, G_i)$ is the open subvariety of representations $\rho \colon \Gamma \to G_i$ such that $\overline{\rho(\Gamma)} = G_i$ (with each $G_i^{\circ}$ semisimple).

Suppose now that $W \subset X_n (\Gamma)$ is an irreducible component of maximal dimension, so that $\dim X_n (\Gamma) = \dim W.$ Then it follows from (\ref{E:D-SS1}) that we can find an irreducible component $V$ of some $R' (\Gamma, G_i)$ such that $\overline{\pi (V)} = W.$ Since $\pi \vert_{V}$ is dominant and separable (as $\mathrm{char} \: K = 0$), it follows from (\cite{Bo}, AG 17.3) that there exists $\rho_0 \in V$ which is a simple point (of $R'(\Gamma, G_i)$) such that $\pi (\rho_0)$ is simple and the differential
\begin{equation}\label{E:D-Diff1}
(d \pi)_{\rho_0} \colon T_{\rho_0} (V) \to T_{\pi (\rho_0)} (W)
\end{equation}
is surjective. Next, let $\psi_{\rho_0} \colon G_i \to R(\Gamma, G_i)$ be the orbit map. By the construction of $\pi$, we have $(\pi \circ \psi_{\rho_0} )(T) = \pi (\rho_0)$ for any $T \in G_i$, so $d(\pi \circ \psi_{\rho_0})_e = 0.$
On the other hand, as we noted in the previous subsection,
the image of the differential $(d \psi_{\rho_0})_e$ is the space $B = B^1 (\Gamma, \tilde{\g}_i)$, where $\tilde{\g}_i$ is the Lie algebra of $G_i$ with $\Gamma$-action given by $\mathrm{Ad} \circ \rho_0$. Since $\rho_0$ is a simple point, it lies on a unique irreducible component of $R'(\Gamma, G_i)$, so it follows that the image of $\psi_{\rho_0}$ (i.e. the orbit of $\rho_0$) is contained in $V$. Consequently, (\ref{E:D-Diff1}) factors through
$$
T_{\rho_0} (V)/ B \to T_{\pi (\rho_0)} (W).
$$
Since obviously $\dim_K T_{\rho_0} (V) \leq \dim_K T_{\rho_0} (\fR(\Gamma, G_i))$ and $T_{\rho_0} (\fR (\Gamma, G_i)) \simeq Z^1 (\Gamma, \tilde{\g})$,
we therefore obtain that
\begin{equation}\label{E:D-TSpDim}
\dim X_n (\Gamma) = \dim W \leq \dim_K H^1 (\Gamma, \tilde{\g}_i).
\end{equation}

So, the proof of Theorem 3 reduces to considering the following situation. Let
$\rho_0 \colon \Gamma \to GL_n (K)$ be a completely reducible representation, set $\G = \overline{\rho_0 (\Gamma)}$ (note that the connected component $\G^{\circ}$ is semisimple), and let $\tilde{\g}$ be the Lie algebra of $\G$, considered as a $\Gamma$-module via $\mathrm{Ad} \circ \rho_0.$
We need to give an upper bound on $\dim_K H^1 (\Gamma, \tilde{\g}).$ This will be made more precise in Proposition \ref{D:P-1} below, after some preparatory remarks.

First, notice that for the purpose of estimating $\dim_K H^1 (\Gamma, \tilde{\g}),$ we may compose $\rho_0$ with the adjoint representation and assume without loss of generality that the group $\G$ is adjoint. Now, since $\G^{\circ}$ is semisimple, $\rho_0$ has a standard description by Theorem 1,
i.e.
there exists a commutative finite-dimensional $K$-algebra $A_0$,
a ring homomorphism
\begin{equation}\label{E:2}
f_0 \colon R \to A_0
\end{equation}
with Zariski-dense image, and a morphism of algebraic groups
\begin{equation}\label{E:3}
\theta \colon G(A_0) \to \G
\end{equation}
such that on a suitable finite-index subgroup $\Delta \subset \Gamma$, we have
\begin{equation}\label{E:1}
\rho_0 \vert_{\Delta} = (\theta \circ F_0) \vert_{\Delta},
\end{equation}
where $F_0 \colon \Gamma \to G(A_0)$ is the group homomorphism induced by $f_0$. Moreover, it follows from Proposition \ref{P:FI-1} that $\theta (G(A_0)) = \G^{\circ}$.

Next, let $\G_1, \dots, \G_r$ be the (almost) simple components of $\G^{\circ}$ (cf. \cite{Bo}, Proposition 14.10). Since $\G^{\circ}$ is adjoint, the product map
$$
\G_1 \times \cdots \times \G_r \to \G^{\circ}
$$
is an isomorphism.
\begin{lemma}\label{D:L-1}
{\rm (i)} The algebraic ring $A_0$ is isomorphic to the product
$
K^{(1)} \times \cdots \times K^{(r)},
$
with $K^{(i)} \simeq K$ for all $i$.

\vskip1mm

\noindent {\rm (ii)} \parbox[t]{16cm}{The morphism $\theta$ is an isogeny. Moreover, for each $i$, $\theta (G(K^{(i)})) = \G_i$, for a unique simple factor of $\G^{\circ}$, and the differential  $(d \theta)_{e} \colon \g \to \tilde{\g}_i$ is an isomorphism of Lie algebras.}
\end{lemma}

\vskip2mm

\noindent Thus, we can write $f_0 \colon R \to A_0$ as
\begin{equation}\label{D:E-hom}
f_0 (t) = (f_0^{(1)} (t), \dots, f_0^{(r)}(t)),
\end{equation}
for some ring homomorphisms $f_0^{(i)} \colon R \to K$.

\vskip2mm



Now, given any ring homomorphism $g \colon R \to K$, we will denote by $\mathrm{Der}^g (R,K)$ the space of $K$-valued derivations of $R$ with respect to $g$, i.e. an element $\delta \in \mathrm{Der}^g (R,K)$
is a map $\delta \colon R \to K$ such that for any $x_1, x_2 \in R$,
$$
\delta (x_1 + x_2) = \delta (x_1) + \delta(x_2) \ \ \ \mathrm{and} \ \ \  \delta (x_1 x_2) = \delta(x_1) g(x_2) + g(x_1) \delta (x_2).
$$


We can now formulate the following proposition.

\begin{prop}\label{D:P-1}
Suppose $\rho_0 \colon \Gamma \to GL_n (K)$ is a linear representation and set $\G = \overline{\rho_0 (\Gamma)}$. Denote by $\tilde{\g}$ the Lie algebra of $\G$ and assume that $\G^{\circ}$ is semisimple. Then
$$
\dim_K H^1 (\Gamma, \tilde{\g}) \leq \sum_{i=1}^r \dim_K \mathrm{Der}^{f_0^{(i)}} (R,K),
$$
where the $f_0^{(i)}$ are the ring homomorphisms appearing in $\mathrm{(\ref{D:E-hom})}$.
\end{prop}

\subsection{Proof of Theorem 3} Assuming Proposition \ref{D:P-1} for now, let us complete the proof of Theorem 3. In view of (\ref{E:D-TSpDim}) and Proposition \ref{D:P-1}, it remains to show that $r \leq n$ and to give a bound on the dimension of the space $\mathrm{Der}^g (R, K)$, for any ring homomorphism $g \colon R \to K,$ which is independent of $g$. Notice that since
$\G^{\circ} \subset GL_n (K)$ and $\G^{\circ} = \G_1 \times \dots \times \G_r,$ we have
$$
n \geq \mathrm{rk} \G^{\circ} = \sum_{i=1}^r \mathrm{rk} \G_i \geq r,
$$
as needed. For the second task, we have the following (elementary) lemma.
\begin{lemma}\label{D:L-Der}
Let $R$ be a finitely generated commutative ring, and denote by $d$ the minimal number of generators of $R$ (i.e. the smallest integer such that there exists a surjection $\Z[x_1, \dots, x_d] \twoheadrightarrow R$). Then for any field $K$ and ring homomorphism $g \colon R \to K$, $\dim_K \mathrm{Der}^g (R, K) \leq d.$ If, moreover, $K$ is a field of characteristic 0, $R$ is an integral domain with field of fractions $L$, and
$g$ is injective, then $\dim_K \mathrm{Der}^g (R, K) \leq \ell,$ where $\ell$ is the transcendence degree of $L$ over its prime subfield.
\end{lemma}


\vskip2mm

\noindent {\bf Remark 6.3.2.} Notice that the estimate $\dim_K \mathrm{Der}^g (R,K) \leq \ell$, where $\ell$ is an in Lemma \ref{D:L-Der}, may not be true if $g$ is not injective. Indeed, take $K = \bar{\Q}$, and let $R_0 = \Z[X,Y]$ and $R = \Z[X,Y]/(X^3 - Y^2).$ Furthermore, let
$$
f \colon \Z[X,Y] \to \bar{\Q}, \ \ \varphi(X,Y) \mapsto \varphi (0,0)
$$
and denote by $g \colon R \to \bar{\Q}$ the induced homomorphism. The space $\mathrm{Der}^f (R_0, \bar{\Q})$ is spanned by the linearly-independent derivations $\delta_x$ and $\delta_y$ defined by
$$
\delta_x (\varphi(X,Y)) = \frac{\partial \varphi}{\partial X} (0,0) \ \ \ \delta_y (\varphi(X,Y)) = \frac{\partial \varphi}{\partial Y} (0,0),
$$
so $\dim_{\bar{\Q}} \mathrm{Der}^f (R_0, \bar{\Q}) = 2$. Now notice that the natural map
$$
\mathrm{Der}^g (R, \bar{\Q}) \to \mathrm{Der}^f (R_0, \bar{\Q})
$$
is bijective. Indeed, it is obviously injective, and since any $\delta \in \mathrm{Der}^f (R_0, \bar{\Q})$ vanishes on the elements of the ideal $(X^3 - Y^2) R_0$, it is also surjective.
Thus, $\dim_{\bar{\Q}} \mathrm{Der}^g (R, \bar{\Q}) = 2.$ On the other hand, if $L$ is the fraction field of $R$, then $\ell := \text{tr. deg.}_{\Q}L$ is 1.


\subsection{Proof of Proposition \ref{D:P-1}} We would now like to indicate the main elements of the proof of Proposition \ref{D:P-1}.

First, let us note that for any finite-index subgroup $\Lambda \subset \Gamma$,
the space of 1-cocycles $Z^1 (\Lambda, \tilde{\g})$ can be naturally identified with the tangent space
\begin{equation}\label{E:D-TSp}
T_{\rho_0}(\fR (\Lambda, \G)) =  \{ \rho \in \mathrm{Hom} (\Lambda, \G(K[\varepsilon])) \mid \mu \circ \rho = \rho_0 \}.
\end{equation}
Also observe that the restriction map
$$
\mathrm{res}_{\Gamma/ \Lambda} \colon H^1 (\Gamma, \tilde{\g}) \to H^1 (\Lambda, \tilde{\g})
$$
is injective.
Indeed, since $[\Gamma : \Lambda] < \infty,$ the corestriction map $\mathrm{cor}_{\Gamma/ \Lambda} \colon H^1 (\Lambda, \tilde{\g}) \to H^1 (\Gamma, \tilde{\g})$ is defined and the composition $\mathrm{cor}_{\Gamma/ \Lambda} \circ \mathrm{res}_{\Gamma/ \Lambda}$ coincides with multiplication by $[\Gamma : \Lambda].$ Since $\mathrm{char} \: K = 0,$ the injectivity of $\mathrm{res}_{\Gamma/ \Lambda}$ follows.

\vskip2mm

\noindent Now, set
$$
X = \mathrm{Der}^{f_0^{(1)}} (R, K) \oplus \cdots \oplus \mathrm{Der}^{f_0^{(r)}} (R, K).
$$
and let $\Delta \subset \Gamma$ be the finite-index subgroup appearing in (\ref{E:1}). In view of the injectivity of the restriction map, to prove Proposition \ref{D:P-1}, it suffices to show that there exists a linear map
$\psi \colon  X \to H^1 (\Delta, \tilde{\g})$ such that
\begin{equation}\label{E:ResInj}
\mathrm{res}_{\Gamma/ \Delta}(H^1 (\Gamma, \tilde{\g})) \subset \mathrm{im} (\psi).
\end{equation}

The map $\psi$ is constructed as follows. Choose derivations $\delta_{i} \in \mathrm{Der}^{f_0^{(i)}} (R,K)$, for $i = 1, \dots, r,$ and let
$$
B = \underbrace{K[\varepsilon] \times \cdots \times K[\varepsilon]}_{r \ \mathrm{copies}}
$$
(with $\varepsilon^2 = 0$).
Then
$$
f_{\delta_1, \dots, \delta_r} \colon R \to B, \ \ \ s \mapsto (f_0^{(1)} (s) + \delta_1 (s) \varepsilon, \dots, f_0^{(r)} (s) + \delta_r (s) \varepsilon)
$$
is a ring homomorphism, hence induces a group homomorphism
$$
F_{\delta_1, \dots, \delta_r} \colon \Gamma \to G(B)
$$
(recall that $\Gamma = G(R)^+ \subset G(R)$).
On the other hand, we have
$$
G(B) \simeq (\g \oplus \cdots \oplus \g) \rtimes (G(K) \times \cdots \times G(K)) \simeq \mathrm{Lie} (G(A_0)) \rtimes G(A_0)
$$
and
$$
\G (K [\varepsilon]) \simeq \tilde{\g} \rtimes \G,
$$
so we can define a group homomorphism $\tilde{\theta} \colon G(B) \to \G (K [\varepsilon])$ by the formula
$$
(x, g) \mapsto ((d \theta)_e (x), \theta (g)),
$$
where $\theta \colon G(A_0) \to \G$ is the morphism appearing in (\ref{E:3}). Notice that using Lemma \ref{D:L-1}(ii), the map
$\tilde{\theta}$ can also be described as follows: let $x_1, \dots, x_r \in \g$ and $g \in G(A_0).$ Then
$$
\tilde{\theta} (x_1, \dots, x_r, g) = \left( \sum_{i=1}^r (d \theta)_e (x_i), \theta (g) \right),
$$
Now, $\tilde{\theta} \circ F_{\delta_1, \dots, \delta_r}$ is a homomorphism $\Gamma \to \G (K[\varepsilon])$, and in view of (\ref{E:1}), we have $$\mu \circ (\tilde{\theta} \circ F_{\delta_1, \dots, \delta_r} \vert_{\Delta}) = \rho_0.$$
It follows from (\ref{E:D-TSp}) that
$$
c_{\delta_1, \dots, \delta_r} := \tilde{\theta} \circ \mathrm{pr} \circ F_{\delta_1, \dots, \delta_r} \vert_{\Delta},
$$
where $\mathrm{pr} \colon G(B) \to \mathrm{Lie}(G(A_0))$ is the projection, is an element of $Z^1 (\Delta, \tilde{\g}).$ Now put
$$
\psi((\delta_1, \dots, \delta_r)) = [c_{\delta_1, \dots, \delta_r}],
$$
where $[c_{\delta_1, \dots, \delta_r}]$ denotes the class of $c_{\delta_1, \dots, \delta_r}$ in $H^1 (\Delta, \tilde{\g}).$

Let us now sketch the proof of the inclusion (\ref{E:ResInj}). Suppose $\rho \colon \Gamma \to \G (K[\varepsilon])$ is a homomorphism with $\mu \circ \rho = \rho_0$. By
Proposition \ref{P:AR-2} and Theorem \ref{T:ARR-1}, we can associate to $\rho$ a commutative finite-dimensional $K$-algebra $A$ together with a ring homomorphism $f \colon R \to A$ with Zariski-dense image. Using the fact that $\rho$ admits a standard description by Theorem 1 (notice that $\overline{\rho (\Gamma)}^{\circ}$ has commutative unipotent radical), we then show that
$$
A \simeq \tilde{K}^{(1)} \times \cdots \times \tilde{K}^{(r)},
$$
where, as above, $r$ is the number of simple components of $\G^{\circ}$, and for each
$i$, $\tilde{K}^{(i)}$ is isomorphic to either $K$ or $K[\varepsilon]$ $($with $\varepsilon^2 = 0).$
So, viewing $A$ as a subalgebra of
$$
\tilde{A} := \underbrace{K[\varepsilon] \times \cdots \times K[\varepsilon]}_{r \ \mathrm{copies}}
$$
we can then write the homomorphism $f \colon R \to A$ in the form
$$
f(t) = (f_0^{(1)} (t) + \delta_1 (t) \varepsilon, \dots, f_0^{(r)} (t) + \delta_{r} (t)\varepsilon)
$$
with $(\delta_1, \dots, \delta_r) \in X$ and $\delta_i = 0$ for $i = r_2+1, \dots, r.$

Now let $\psi((\delta_1, \dots, \delta_r)) = d_{\delta_1, \dots, \delta_r} \in Z^1 (\Delta, \tilde{\g})$ and let $c_{\rho}$ be the element of $Z^1(\Gamma, \tilde{\g})$ corresponding to $\rho.$ A further analysis of the standard description of $\rho$ shows that
$$
\mathrm{res}_{\Gamma / \Delta} ([c_{\rho}]) = [d_{\delta_1, \dots, \delta_r}]
$$
as elements of $H^1 (\Delta, \tilde{\g})$. From this (\ref{E:ResInj}) follows.
$\hfill \Box$

\vskip5mm

\subsection{Applications to rigidity} The final thing that we would like to discuss is
how Theorems 1 and 3 can be used to obtain various forms of classical rigidity for the elementary groups $G(\I)^+$, where $G$ is a universal Chevalley-Demazure group scheme corresponding to a reduced irreducible root system $\Phi$ of rank $> 1$ and $\I$ is a ring of algebraic integers (or $S$-integers) in a number field. It is worth mentioning that all of the rigidity statements we discuss here ultimately boil down to the fact that $\I$ does not admit nontrivial derivations.

To fix notations, let $\Phi$ be a reduced irreducible root system of rank $>1$, $G$ the universal Chevalley-Demazure group scheme of type $\Phi$, and $\I$ a ring of algebraic $S$-integers in a number field $L$ such that $(\Phi, \I)$ is a nice pair. 
Furthermore, let $\Gamma = G(\I)^+$ be the elementary subgroup of $G(\I)$ (in fact, here $\Gamma$ actually coincides with $G(\I)$ (see \cite[Lemma 4]{Tav1}), but we will not need this). The following proposition gives a \emph{superrigidity} result in the present context. The more classical version of superrigidity will be discussed in Example 6.5.3 below.

\begin{prop}\label{SR:P-1}
Let $\rho \colon \Gamma \to GL_m (K)$ be an abstract linear representation over an algebraically closed field $K$ of characteristic 0. Then there exist

\vskip1mm

\noindent {\rm (i)} \parbox[t]{16cm}{a finite dimensional commutative $K$-algebra
$$
A \simeq K^{(1)} \times \cdots \times K^{(r)},
$$
with $K^{(i)} \simeq K$ for all $i$;}

\vskip1mm

\noindent {\rm (ii)} \parbox[t]{15cm}{a ring homomorphism $f = (f^{(1)}, \dots, f^{(r)}) \colon \I \to A$ with Zariski-dense image, where each $f^{(i)} \colon \I \to K^{(i)}$ is the restriction to $\I$ of an embedding $\varphi_i \colon L \hookrightarrow K$, and $\varphi_1, \dots, \varphi_r$ are all \emph{distinct}; and}

\vskip1mm

\noindent {\rm (iii)} \parbox[t]{13cm}{a morphism of algebraic groups $\sigma \colon G(A) \to GL_m (K)$}

\vskip1mm

\noindent such that for a suitable subgroup
of finite index $\Delta \subset \Gamma$, we have
$$
\rho \vert_{\Delta} = \sigma \vert_{\Delta}.
$$
\end{prop}
\begin{proof}
Let $H = \overline{\rho (\Gamma)}$ be the Zariski closure of the image of $\rho.$
We begin by showing that under our assumptions, the connected component $H^{\circ}$ is automatically reductive. Suppose this is not the case and let $U$ be the unipotent radical of $H^{\circ}.$ Since the commutator subgroup $U' = [U,U]$ is a closed normal subgroup of $H$, the quotient $\check{H} = H/U'$ is affine, so we have a closed embedding $\iota \colon \check{H} \to GL_{m'} (K)$ for some $m'$.
Then, $\check{\rho} = \iota \circ \pi \circ \rho,$ where $\pi \colon H \to \check{H}$ is the quotient map, is a linear representation of $\Gamma$ such that $\overline{\check{\rho} (\Gamma)}^{\circ} = \check{H}^{\circ}$ has commutative unipotent radical. So, from Theorem 1, we obtain
a finite-dimensional commutative
$K$-algebra $\check{A}$, a ring homomorphism $\check{f} \colon \I \to \check{A}$ (which is injective as any nonzero ideal in $\I$ has finite index) with Zariski-dense image, and a morphism $\check{\sigma} \colon G(\check{A}) \to \check{H}$ of algebraic groups such that for a suitable finite-index subgroup $\check{\Delta} \subset \Gamma$, we have
$$
\check{\rho} \vert_{\check{\Delta}} = (\check{\sigma} \circ \check{F})\vert_{\check{\Delta}},
$$
where $\check{F} \colon \Gamma \to G(\check{A})$ is the group homomorphism induced by $\check{f}.$

Now let $J$ be the Jacobson radical of $\check{A}$. Since $\check{H}^{\circ}$ has commutative unipotent radical, $J^2 = \{ 0 \}$ by Lemma \ref{L:FI-1}. We claim that in fact $J = \{ 0 \}.$
Indeed, using the Wedderburn-Malcev Theorem as in the discussion preceding Proposition \ref{P:R-3},
we can write $\check{A} = \oplus_{i=1}^r \check{A}_i$, where for each $i$, $\check{A}_i = K \oplus J_i$ is a finite-dimensional local $K$-algebra with maximal ideal $J_i$ such that $J_i^2 = \{ 0 \}.$ Then it suffices to show that $J_i = \{ 0 \}$ for all $i$.
So, we may assume that $\check{A}$ is itself a local $K$-algebra of this form. Then, fixing a $K$-basis $\{\varepsilon_1, \dots, \varepsilon_s \}$ of $J$, we have
$$
\check{f}(x) = f_0(x) + \delta_1 (x) \varepsilon_1 + \cdots + \delta_s (x) \varepsilon_s,
$$
where $f_0 \colon \I \to K$ is an injective ring homomorphism and $\delta_1, \dots, \delta_s \in \mathrm{Der}^{f_0} (\I, K).$
On the other hand, since the fraction field of $\I$ is a number field, it follows from Lemma \ref{D:L-Der} that the
derivations $\delta_1, \dots, \delta_s$ are identically zero. So, the fact that
$\check{f}$ has Zariski-dense image forces $J = \{ 0 \}.$ Consequently,
$\check{A} \simeq K \times \cdots \times K.$

Next, Proposition \ref{P:FI-1} implies that $\check{\sigma} \colon G(\check{A}) \to \check{H}^{\circ}$ is surjective,
so $\check{H}^{\circ}$ is semisimple, in particular reductive (\cite{Bo}, Proposition 14.10). It follows that $U = [U, U]$ (cf. \cite{Bo}, Corollary 14.11), and hence, being a nilpotent group, $U = \{ e \}$, which contradicts our original assumption. Thus, $H^{\circ}$ must be reductive, as claimed.

We can now apply Theorem 1 to $\rho$ to obtain
a finite-dimensional commutative $K$-algebra $A$, a ring homomorphism $f \colon \I \to A$ with Zariski-dense image, and a morphism $\sigma \colon G(A) \to H$ of algebraic groups such that for a suitable subgroup of finite index $\Delta \subset \Gamma$, we have
$$
\rho \vert_{\Delta} = (\sigma \circ F) \vert_{\Delta}.
$$
Moreover, the fact that $H^{\circ}$ is reductive implies that
$A = K \times \cdots \times K$ (cf. Proposition \ref{P:AR-20} and Lemma \ref{L:FI-1}). So, we can write
$f = (f^{(1)}, \dots, f^{(r)})$, for some ring homomorphisms $$f^{(1)}, \dots, f^{(r)} \colon \I \to K.$$
It is easy to see that all of the $f^{(i)}$ are injective,
and since $L$ is the fraction field of $\I$, it follows that each homomorphism $f^{(i)}$ is a restriction to $\I$ of an embedding  $\varphi_i \colon L \hookrightarrow K.$ Finally, since $f$ has Zariski-dense image,
all of the $\varphi_i$ must be {\it distinct}. This completes the proof.
\end{proof}

\vskip2mm

\noindent Keeping the notations of the proposition, we have the following

\begin{cor}\label{SR:C-1}
Any representation $\rho \colon \Gamma \to GL_m (K)$ is completely reducible.
\end{cor}
\begin{proof}
By Proposition \ref{SR:P-1}, we have $\rho \vert_{\Delta} = \sigma \vert_{\Delta},$ so since $G(B)$ is a semisimple group and $\mathrm{char} \: K = 0,$ $\rho \vert_{\Delta}$ is completely reducible. Since $\Delta$ is a finite-index subgroup of $\Gamma$, it follows that $\rho$ is also completely reducible.
\end{proof}

\vskip5mm

\noindent {\bf $SS$-rigidity and local rigidity.} Notice that since by Lemma \ref{D:L-Der} there are no nonzero derivations $\delta \colon \I \to K$, Proposition \ref{D:P-1} and the estimate given in (\ref{E:D-TSpDim}) yield that for $\Gamma = G(\I)^+,$ we have $\dim X_n (\Gamma) = 0$ for all $n \geq 1$, i.e.
$\Gamma$ is $SS$-rigid. In fact, Corollary \ref{SR:C-1} implies that
$\Gamma$
is {\it locally rigid}, that is
$H^1 (\Gamma, \mathrm{Ad} \circ \rho) = 0$ for {\it any} representation $\rho \colon \Gamma \to GL_m (K)$.
This is shown in \cite{LM}, and we recall the argument for the reader's convenience. Let
$V = K^m.$ It is well-known that
$$
H^1 (\Gamma, \mathrm{End}_K (V,V)) = \mathrm{Ext}_{\Gamma}^1 (V,V)
$$
(cf. \cite{LM}, pg. 37), and
$\mathrm{Ext}_{\Gamma}^1 (V,V) = 0$ by Corollary \ref{SR:C-1}. But $\mathrm{Ad} \circ \rho$, whose underlying vector space is $M_m (K)$, can be naturally identified as a $\Gamma$-module with $\mathrm{End}_K (V,V),$ so $H^1 (\Gamma, \mathrm{Ad} \circ \rho) = 0$, as claimed.

\vskip4mm

\noindent {\bf Example 6.5.3.} (cf. \cite{BMS}, \S 16, and \cite{Mar}, Ch. VII). \ Let $\Gamma = SL_n (\Z)$ ($n \geq 3$) and consider an abstract representation $\rho \colon \Gamma \to GL_m (K).$ Then there exists a rational representation $\sigma \colon SL_n (K) \to GL_m (K)$ such that
$$
\rho \vert_{\Delta} = \sigma \vert_{\Delta}
$$
for a suitable finite-index subgroup $\Delta \subset \Gamma.$ Indeed, let $f \colon \Z \to A$ be the homomorphism associated to $\rho$. Since $A \simeq K^{(1)} \times \cdots \times K^{(r)}$ by Proposition \ref{SR:P-1}, we see that
$f$ is simply a diagonal embedding of $\Z$ into $K \times \cdots \times K.$ But $f$ has Zariski-dense image, so $r = 1,$ and the rest follows.

Notice that for a general ring of $S$-integers $\I$, the algebraic group $G(A)$ that arises in Proposition \ref{SR:P-1} can be described as follows. Let $\G = R_{L/\Q} (_{L}\!G)$, where $_{L}\!G$ is the algebraic group obtained from $G$ by extending scalars from $\Q$ to $L$ and $R_{L/\Q}$ is the functor of restriction of scalars. Then $\G(K) \simeq G(K) \times \cdots \times G(K)$, with the factors corresponding to all of the distinct embeddings of $L$ into $K$ (cf. \cite{PR}, \S 2.1.2). The group $G(A)$ is then obtained from $\G(K)$ by simply projecting to the factors corresponding to the embeddings $\varphi_1, \dots, \varphi_r$, so any representation of $G( \I)^+$ factors through $\G.$

\vskip5mm

\bibliographystyle{amsplain}

\end{document}